\documentclass[final,leqno]{siamltex704}
\usepackage{amsmath}
\usepackage{graphicx}
\usepackage[notcite,notref]{showkeys}
\usepackage{mathrsfs}
\usepackage{appendix}
\usepackage{color}
\usepackage{float}
\usepackage{amsfonts,amssymb}
\usepackage{dsfont}
\usepackage{pifont}
\usepackage{wrapfig} % Allows in-line images
\usepackage{hyperref}
\usepackage{multirow}
\numberwithin{equation}{section}
\def\3bar{{|\hspace{-.02in}|\hspace{-.02in}|}}
\def\E{{\mathcal{E}}}
\def\T{{\mathcal{T}}}

\def\bb{{\mathbf{b}}}

\def\pT{{\partial T}}

\def\W{{\mathcal{W}}}

\def\CQ{{\mathcal{Q}}}

\def\bx{{\mathbf{x}}}

\def\bn{{\mathbf{n}}}

\newtheorem{remark}{Remark}[section]
\newtheorem{algorithm}{Primal-Dual Weak Galerkin Algorithm}[section]

\title {Low Regularity Primal-Dual Weak Galerkin Finite Element Methods for  Convection-Diffusion Equations}

\author{Chunmei Wang \thanks{Department of Mathematics \& Statistics, Texas Tech University, Lubbock, TX 79409, USA (chunmei.wang@ttu.edu). The research of Chunmei Wang was partially supported by National Science Foundation Award DMS-1849483.} %\and Junping Wang\thanks{Division of Mathematical
  \and Ludmil Zikatanov  \thanks{Department of Mathematics, Penn State University, University Park, PA, 16802, USA (ludmil@ psu.edu).
    The work of Zikatanov is supported in part by NSF
    DMS-1720114 and DMS-1819157.    
}}

\begin{document}

\maketitle

\begin{abstract}
  We propose a numerical method for convection-diffusion problems under
  low regularity assumptions. We derive the method and analyze it
  using the primal-dual weak Galerkin (PDWG) finite element
  framework. The Euler-Lagrange formulation resulting from the
  PDWG scheme yields a system of equations
  involving not only the equation for the primal variable but
  also its adjoint for the dual variable. We show that the proposed
  PDWG method is stable and convergent. We also derive a priori error
  estimates for the primal variable
  in the 
  $H^{\epsilon}$-norm for $\epsilon\in [0,\frac12)$.  A series of
    numerical tests that validate the theory are presented as
    well.
\end{abstract}



\pagestyle{myheadings}
\section{Introduction}
In this paper we consider the model
convection diffusion problem
for an unknown function $u$ satisfying
\begin{equation}\label{model}
\begin{split}
-\nabla\cdot (a \nabla u) +\nabla\cdot (\bb u)=&f, \ \qquad
\text{in}\quad \Omega,\\ u=&g_1, \qquad \text{on}\quad \Gamma_D,\\ (-a
\nabla u +\bb u)\cdot \bn=&g_2, \qquad \text{on}\quad \Gamma_N.
\end{split}
\end{equation}
Here, $\Omega\subset \mathbb R^d (d=2, 3)$ is an open bounded domain
whose boundary $\partial \Omega$ is a Lipschitz polyhedron (polygon
for $d=2$) with $\Gamma_D \cup \Gamma_N=\partial \Omega$.  Further, $\bn$
is the unit outward normal direction to $\Gamma_N$. We assume that the
convection vector $\bb \in [W^{1, \infty}(\Omega)]^d$ is bounded, and
the diffusion tensor $a=\{a_{ij}\}_{d\times d}$ is symmetric and
positive definite; i.e., there exists a constant $\alpha>0$, such that
$$
\xi^T a\xi\ge \alpha \xi^T\xi, \qquad \forall \xi \in  \mathbb R^d.
$$
In addition, we assume that the diffusion tensor $a$ and the
convection vector $\bb$ are uniformly piecewise continuous functions.

%% The convection-diffusion equation, one of the most important
%% mathematical models in engineering and physics, is a combination of
%% the diffusion equation and the convection (advection) equation. The
%% convection-diffusion equation describes physical phenomena where
%% particles, energy, or other physical quantities are transferred inside
%% a physical system due to two processes: diffusion and convection. In
%% many of the applications, the size of the diffusion tensor $a$ is much
%% smaller in several orders of magnitude compared with the size of the
%% convection vector $\bb$. The very small diffusion tensor $a$ leads to
%% the convection-dominated diffusion equation which is a singularly
%% perturbed problem and is thus quite challenging to solve numerically
%% in the sense that the numerical solution of this boundary value
%% problem typically possesses layers (internal and/or boundary), which
%% are thin regions where the solution and/or its derivatives change
%% rapidly.

As is well known~\cite{er2002} the standard Galerkin finite element
approximation for the convection-diffusion often exhibit nonphysical
oscillations, especially when the convection is dominating (i.e. the
eigenvalues of $a$ are small compared to the size of $\bb$) and they
also do not provide accurate approximations unless the mesh size is
sufficiently small.  A variety of numerical stabilization techniques
have been developed to resolve this challenge in the past several
decades such as fitted mesh methods \cite{mos1996, rst2008}, fitted
operator methods \cite{rst2008}, and the methods using approximations of the fluxes~\cite{xl1999, lazarovz2004}.  Such methods are dominated by
  upwind-type schemes and are applicable to the problems of
  complicated domains or layer structures. The upwind-type schemes
  were first proposed in the finite difference methods, and later were
  extended to finite element methods. The key idea in the upwind
  methods is to obtain a stabilized discretization method by adding an
  artificial diffusion/viscosity to balance the convection term. Among
  the various schemes, the streamline upwind Petrov-Galerkin method
  proposed by Hughes and Brooks is an efficient numerical method
  \cite{hb1979, bh1982} in improving the stability of the standard
  Galerkin method through the use of an additional stabilization term
  in the upwind direction to suppress most of the nonphysical
  oscillations while keeping the accuracy. However, one disadvantage
  in the upwind methods lies in that too much artificial diffusion
  leads to smearing layers, especially for the problems in multiple
  dimensions. Bakhvalov \cite{b1969} proposed the optimization of
  numerical meshes, where the meshes were generated from projections
  of an equidistant partition of layer functions. Another effective
  idea of piecewise-equidistant meshes was proposed by Shishkin
  \cite{shishkin1990}. The adaptive method has been proposed
  \cite{b1976, boor1973} to address a variety of difficulties
  including layers \cite{reinhardt1981}. The discontinuous Galerkin
  (DG) method \cite{b1973, dd1975, rh1973, wheeler1978, abcm2002} is
  an effective technique for solving conservation laws for elliptic
  problems. Furthermore, DG schemes include a upwinding which is
  equivalent to the stabilization for the convection-diffusion
  problems \cite{johnson1987, rst2008, am2009, bo1999, bhs2006, guzman2006, hss2002, hsbb2006, hw2008}.  Recently, Burman and He \cite{burmanhe} developed a primal dual mixed finite element method for indefinite advection-diffusion equations with optimal a priori error estimates in the energy and the $L^2$ norm for the primal variable when the Pecket number is low. In \cite{burman}, Burman, Nechita and Oksanen devised a stabilized finite element method for a kind of inverse problems subject to the convection-diffusion equation in the diffusion-dominated regime. Some error estimates in local $H^1$ or $L^2$ norms were derived for their numerical approximations.

The goal of this paper is to derive and analyze a finite element
discretization for the convection-diffusion problem (\ref{model}). We
have chosen primal-dual weak Galerkin (PDWG) framework for such
derivation and analysis. The PDWG was introduced and successfully used for the numerical solution of elliptic Cauchy problems~\cite{w2018,ww2018}, elliptic equations in non-divergence
form~\cite{ww2016}, and Fokker-Planck equations~\cite{ww2017}.  The idea of PDWG is to enhance the stability and solvability of the numerical solutions by a combined consideration of the primal and the dual/adjoint equation through the use of properly defined stabilizers or smoothers. The very similar idea has been developed and utilized by Burman and his collaborators in \cite{burman2013, burman2014, burman2014-2, burman2016, burman2017, burman2018, burmanhe, burman2018-2, burman} in other finite element contexts. This choice is motivated by the facts that the PDWG techniques are natural for deriving error estimates under low regularity assumptions, and also allow for general polyhedral (not necessarily simplicial) finite
elements. Methods for convection-diffusion equations on such general
meshes have been also developed in the context of Virtual Finite
Element methods (VEM)~\cite{an,bbm,irisarri}, DG methods and
Hybrydized DG methods
(HDG)~\cite{chsszz,chen2012,ghmszz,hszz,cdgrs}. While in many cases
variants of the HDG, VEM and WG methods are shown to be
equivalent~\cite{cockburn,chen,hongxu}, for low regularity solutions
such equivalences are not helping much in the error estimates. Our
analysis here shows that the PDWG method essentially can assumes very
low regularity for the primal problem so that solutions with
discontinuities can be approximated well and further allow us to give
a priori estimates for the primal variable in $H^{\epsilon}$-norm for
$0\le \epsilon< \frac12$.

%% Compared with the standard Galerkin finite element method, our
%% analysis of error estimates is based on a low
%% $H^{2-\epsilon}(\Omega)\ (0\leq\epsilon<\frac{1}{2}) $ regularity
%% assumption for the solution of the dual problem.
%% In such a way we only need to assume
%% $u\in H^\epsilon(\Omega)$ for the primal variable.
%% Consequently, in the
%% our PD-WG method, the Dirichlet boundary value is required to satisfy
%% $g_1\in L^2(\Gamma_D)$ which is much weaker than the usual condition
%% of $g_1\in H^{\frac{1}{2}}(\Gamma_D)$ in the standard Galerkin finite
%% element method. Furthermore, the Dirichlet boundary data $g_1$ in our
%% PD-WG method appears in the boundary integral which greatly reduces
%% possible noises in the applications.  The distinguished
%% feature/advantage of this PD-WG method lies on ultra weak regularity
%% assumptions for the primal variable $u$ in the mathematical
%% convergence theory since all the derivatives are applied only to the
%% dual variable or, equivalently, to the test function in
%% \eqref{weakform}. 

The paper is organized as follows. Section 2 is devoted to a
discussion/review of the weak differential operators as well as their
discretizations. In Section 3, the primal-dual weak Galerkin algorithm
for the convection-diffusion problem (\ref{model}) is
proposed. Section 4 presents some technical results, including the
critical {\em inf-sup} condition, which plays an important role in
deriving the error analysis in Section 6. The error equations for the
PD-WG scheme are derived in Section 5. In Section 6, the error
estimates in an optimal order are derived for the primal-dual WG
finite element method in some discrete Sobolev norms. Finally in
Section 7, a series of numerical results are reported to demonstrate
the effectiveness and accuracy of the numerical method developed in
the previous sections.

\section{Preliminaries and notation}
Throughout the paper, we follow the usual notation for Sobolev spaces
and norms. For any open bounded domain $D\subset \mathbb{R}^d$ with
Lipschitz continuous boundary, we use $\|\cdot\|_{s,D}$ and
$|\cdot|_{s,D}$ to denote the norm and seminorm in the Sobolev space
$H^s(D)$ for any $s\ge 0$, respectively. The norms in $H^{s}(D)$ for $s<0$ are defined by duality with the norms in $H^{|s|}(D)$.
The inner product in $H^s(D)$
is denoted by $(\cdot,\cdot)_{s,D}$. The space $H^0(D)$ coincides with
$L^2(D)$, for which the norm and the inner product are denoted by
$\|\cdot \|_{D}$ and $(\cdot,\cdot)_{D}$, respectively. When
$D=\Omega$, or when the domain of integration is clear from the
context, we drop the subscript $D$ in the norm and the inner product
notation. For convenience, throughout the paper, we use ``$\lesssim$
'' to denote ``less than or equal to up to a general constant
independent of the mesh size or functions appearing in the
inequality".

We begin by introducing the weak formulation of the
convection-diffusion model problem (\ref{model}).
After integration by parts the variational problem can
be stated as follows: Find $u\in L^2(\Omega)$
satisfying
\begin{equation}\label{weakform}
 (u,  \nabla \cdot (a\nabla w)+\bb  \cdot \nabla w) =-(f, w)+\langle g_2, w\rangle_{\Gamma_N}+\langle g_1, a\nabla w \cdot \bn\rangle_{\Gamma_D}, \quad \forall w\in W,
\end{equation}
where $W=\{w\in H^1(\Omega), a\nabla w \in H(div; \Omega),  w|_{\Gamma_D}=0, a\nabla w \cdot \bn|_{\Gamma_N}=0\}.$

The dual problem corresponding to this primal formulation
then is: For a given $\psi\in H^{-\epsilon}(\Omega)$,
find $\lambda \in W$
such that $\lambda|_{\Gamma_D}=0$, 
$a\nabla\lambda\cdot\bn|_{\Gamma_N}=0$, and
\begin{equation}\label{EQ:dual-form}
  (v, \nabla \cdot(a \nabla \lambda)+\bb \cdot \nabla \lambda) = (\psi,v),
  \qquad \forall v\in H^{\epsilon}(\Omega),
\end{equation}
In the following we assume that the solution to the this dual problem is $H^{2-\epsilon}(\Omega)$-regular for some $\epsilon\in [0,\frac12)$ and satisfies the reqularity estimate
  \begin{equation}\label{EQ:H2Regularity1}
  \|\lambda\|_{2-\epsilon}\lesssim \|\psi\|_{-\epsilon}.
  \end{equation}  
  This reqularity assumption also implies that
  when $\psi\equiv 0$, then {\em the dual problem \eqref{EQ:dual-form} has
    one and only one solution, namely the trivial solution
    $\lambda\equiv 0$}.

  \begin{remark}
  Notice that the primal and the dual equations are unrelated to each
  other in the continuous model. However, combining the discrete
  primal equation with the discrete dual equation through some
  stabilization terms in the context of weak Galerkin finite element
  methods gives rise to an efficient numerical scheme.
\end{remark}

\section{Discrete Weak Differential Operators}\label{Section:Hessian}

Denote by ${\cal L}:=\nabla \cdot (a\nabla)$ the diffusive operator in
\eqref{model}. The operator ${\cal L}$ and the gradient operator are
the two principle differential operators used in the weak formulation
(\ref{weakform}) for the convection-diffusion equation
(\ref{model}). This section will introduce a weak version of ${\cal
  L}$ and the gradient operator; see \cite{wy3655} for more
information.

Let ${\cal T}_h$ be a finite element partition of the domain $\Omega$
into polygons in 2D or polyhedra in 3D which is \emph{shape regular}. By shape regularity here we mean a partition such that
for any $T\in {\cal T}_h$, $\phi\in H^{1-\theta}(T)$, and $\theta\in [0,\frac12)$ the following trace inequality holds
\begin{equation}\label{tracein}
 \|\phi\|^2_{\partial T} \lesssim h_T^{-1}\|\phi\|_T^2+h_T^{1-2\theta} \|\phi\|_{1-\theta,T}^2, 
\end{equation}
and, in addition, when $\phi$ is a polynomial on the element $T\in {\cal T}_h$,
we also have the inverse inequality 
\begin{equation}
 \|\phi\|^2_{\partial T} \lesssim h_T^{-1}\|\phi\|_T^2.
\end{equation}
We refer the reader to \cite{wy3655} for details and
discussion of sufficient conditions on the
partition so that these inequalities hold.

Further, we denote by ${\mathcal E}_h$ the set of all edges or flat
faces in ${\cal T}_h$ and ${\mathcal E}_h^0={\mathcal E}_h \setminus
\partial\Omega$ the set of all interior edges or flat faces. Denote by
$h_T$ the meshsize of $T\in {\cal T}_h$ and $h=\max_{T\in {\cal
    T}_h}h_T$ the meshsize of ${\cal T}_h$.

Let $T\in \mathcal{T}_h$ be a polygonal or polyhedral region with
boundary $\partial T$. A weak function on $T$ refers is the triplet
$v=\{v_0,v_b, v_n\}$ such that $v_0\in L^2(T)$, $v_b\in L^{2}(\partial
T)$ and $v_n\in L^{2}(\partial T)$. The first two components of $v$,
namely $v_0$ and $v_b$, can be understood as the value of $v$ in the
interior and on the boundary of $T$, respectively. The third component
$v_n$ can be interpreted as the value of $a\nabla v \cdot \bn$ on the
boundary $\partial T$, where and in what follows of this paper $\bn$
denotes the outward normal vector on $\pT$. Note that $v_b$ and $v_n$
may not necessarily be related to the traces of $v_0$ and $a\nabla v_0
\cdot \bn$ on $\partial T$. But taking $v_b$ as the trace of $v_0$ on
$\partial T$ and/or $v_n$ as the trace of $a\nabla v_0 \cdot \bn$ on
$\partial T$ is not prohibited in the forthcoming discussion and
applications.

Let $\W(T)$ be the space defined as
\begin{equation*}\label{2.1}
\W(T)=\{v=\{v_0,v_b, v_n \}: v_0\in L^2(T), v_b\in L^{2}(\partial T), v_n\in L^{2}(\partial T)\}.
\end{equation*}
The weak gradient of $v\in \W(T)$, denoted by $\nabla_w v$, is the linear functional in $[H^1(T)]^d$ such that
\begin{equation*}
(\nabla_w v, \boldsymbol{\psi})_T=-(v_0,\nabla \cdot
  \boldsymbol{\psi})_T+\langle v_b, \boldsymbol{\psi} \cdot
  \textbf{n}\rangle_{\partial T},
\end{equation*}
for all $\boldsymbol{\psi}\in [H^1(T)]^d$.
The weak action of ${\cal L}=\nabla \cdot (a\nabla)$ on $v\in \W(T)$, denoted by  ${\cal L}_w v$, is defined as a linear functional in $H^2(T)$ such that
\begin{equation*}
({\cal L}_w v, w)_T= (v_0, {\cal L} w)_T-\langle v_b, a \nabla w\cdot
\textbf{n}\rangle_{\partial T}+ \langle v_n,w\rangle_{\partial T},
\end{equation*}
for all $w \in  H^2(T)$.

Denote by $P_r(T)$ the space of polynomials on $T$ with degree $r$ and less. A discrete version of $\nabla_{w} v$  for $v\in \W(T)$, denoted by $\nabla_{w,r,T}v$, is defined as the unique vector-valued polynomial in  $[P_r(T) ]^d$ satisfying
\begin{equation}\label{disgradient}
(\nabla_{w,r,T}  v, \boldsymbol{\psi})_T=-(v_0,\nabla \cdot
\boldsymbol{\psi})_T+\langle v_b, \boldsymbol{\psi} \cdot
\textbf{n}\rangle_{\partial T}, \quad\forall\boldsymbol{\psi}\in
[P_r(T)]^d,
\end{equation}
which, from the usual integration by parts, gives
\begin{equation}\label{disgradient*}
(\nabla_{w,r,T}  v, \boldsymbol{\psi})_T= (\nabla v_0, \boldsymbol{\psi})_T-\langle v_0- v_b, \boldsymbol{\psi} \cdot \textbf{n}\rangle_{\partial T}, \quad\forall\boldsymbol{\psi}\in [P_r(T)]^d.
\end{equation}

A discrete version of ${\cal L}_w v$ for $v\in \W(T)$, denoted by ${\cal L}_{w,r,T} v$, is defined as the unique polynomial in $P_r(T)$ satisfying
\begin{equation}\label{disvergence}
({\cal L}_w v, w)_T= (v_0, {\cal L} w)_T-\langle v_b, a \nabla w\cdot
\textbf{n}\rangle_{\partial T}+ \langle v_n,w\rangle_{\partial T}, \quad\forall w \in P_r(T),
\end{equation}
which, from the usual integration by parts, gives
\begin{equation}\label{disvergence*}
({\cal L}_w v, w)_T= ({\cal L}v_0, w)_T+\langle v_0-v_b, a \nabla w\cdot
\textbf{n}\rangle_{\partial T}- \langle a\nabla v_0\cdot \bn-v_n, w \rangle_{\partial T},  \forall w \in P_r(T).
\end{equation}

\section{Primal-Dual Weak Galerkin Formulation}\label{Section:WGFEM}
For any given integer $k\geq 1$, denote by $W_k(T)$ the discrete
function space given by
$$
W_k(T)=\{\{\sigma_0,\sigma_b, \sigma_n\}:\sigma_0\in P_k(T),\sigma_b\in
P_k(e), \sigma_n \in P_{k-1}(e),e\subset \partial T\}.
$$
By patching $W_k(T)$ over all the elements $T\in {\cal T}_h$
through a common value $v_b$ on the interior interface $\E_h^0$, we arrive at a global weak finite element space $W_h$; i.e.,
$$
W_h=\big\{\{\sigma_0,\sigma_b, \sigma_n\}:\{\sigma_0,\sigma_b, \sigma_n\}|_T\in
W_k(T), \forall T\in {\cal T}_h
 \big\}.
$$
Denote by $W_h^0$ the subspace of $W_h$ with vanishing boundary value; i.e.,
$$
W_h^0=\{v\in W_h: v_b=0 \ \text{on}\ \Gamma_D, v_n=0 \ \text{on}\ \Gamma_N\}.
$$
Denote by $M_h$ the space of piecewise polynomials of degree $s$; i.e.,
$$
M_h=\{w: w|_T\in P_{s}(T),  \forall T\in {\cal T}_h\}.
$$
Here, the integer $s$ is taken as either $k-1$ or $k-2$, as appropriate. When it comes to the case of the lowest order (i.e., $k=1$), the only option is $s=0$.

For simplicity of notation and without confusion,  for any $\sigma\in
W_h$, denote by $\nabla_{w}\sigma$ the discrete weak gradient
$\nabla_{w,k-1,T}\sigma$ computed  by
(\ref{disgradient}) on each element $T$; i.e.,
$$
(\nabla_{w}\sigma)|_T= \nabla_{w, k-1, T}(\sigma|_T), \qquad \sigma \in W_h.
$$
Similarly, for any $\sigma\in W_h$, denote by ${\cal L}_{w}\sigma$ the discrete weak-${\cal L}$ operation ${\cal L}_{w, s, T}\sigma$ computed by (\ref{disvergence}) on each element $T$; i.e.,
$$
({\cal L}_{w}\sigma)|_T={\cal L}_{w, s, T}(\sigma|_T), \qquad \sigma\in W_h.
$$

For any $ \lambda, w \in W_h$, and $u\in M_h$, we introduce the following bilinear forms
\begin{eqnarray}
s(\lambda, w)&=&\sum_{T\in {\cal T}_h} s_T(\lambda, w),\label{EQ:s-form}\\
b(u, w)&=&\sum_{T\in {\cal T}_h}(u, {\cal L}_w w+\bb \cdot \nabla_w w)_T,\label{EQ:b-form}
\end{eqnarray}
where

\begin{equation}\label{EQ:sT-form}
\begin{split}
s_T(\lambda, w)=&h_T^{-3}\langle (|a|_T+|\bb \cdot \bn|) (\lambda_0-\lambda_b), w_0-w_b\rangle_{\partial T}
\\&+h_T^{-1}\langle  a \nabla \lambda_0 \cdot \bn-\lambda_n,  a \nabla w_0 \cdot \bn-w_n\rangle_{\partial T}\\
&+\gamma ({\cal L} \lambda_0+\bb \cdot \nabla \lambda_0,  {\cal L} w_0+\bb \cdot \nabla w_0)_T.\\
\end{split}
\end{equation}
Here, $\gamma \ge 0$ is a parameter independent of the meshsize $h$ and the functions involved; and $|a|_T=\sup_{\bx\in T}(\sum_{i, j=1}^d a^2_{ij}(\textbf{x}))^{\frac{1}{2}}$.

Let $k\ge 2$ be a given integer and $s\ge 0$ be another integer. Our primal-dual weak Galerkin finite element method for the convection-diffusion model problem (\ref{model}) based on the variational formulation (\ref{weakform}) can be described as follows:
\begin{algorithm} Find $(u_h;\lambda_h)\in M_h
\times W_{h}^0$ satisfying
\begin{eqnarray}\label{32}
s(\lambda_h, w)+b(u_h,w)&=&-(f,w_0)+ \langle g_2, w_b \rangle_{\Gamma_N}+\langle g_1, w_n\rangle_{\Gamma_D},\quad  \forall w\in W^0_h,\\
b(v, \lambda_h)&=& 0,  \qquad\qquad\qquad\qquad\qquad\qquad\qquad\qquad \forall v\in M_{h}.\label{2}
\end{eqnarray}
\end{algorithm}

For any $w\in H^1(\Omega)$, denote by $Q_h w$ the $L^2$ projection onto the weak finite element space $W_h$ such that on each element $T$,
$$
Q_hw=\{Q_0w,Q_bw, Q_n(a\nabla w \cdot \bn)\}.
$$
Here and in what follows of this paper, on each element $T$, $Q_0$ denotes the $L^2$ projection operator onto $P_k(T)$; on each edge or face $e\subset\partial T$, $Q_b$ and $Q_n$ stand for the $L^2$ projection operators onto $P_{k}(e)$ and $P_{k-1}(e)$, respectively. Denote by ${\cal Q}^{k-1}_h$ and ${\cal Q}^{s}_h$ the $L^2$ projection operators onto the space of piecewise vector-valued polynomials of degree $k-1$ and the space $M_h$, respectively.

\begin{lemma}\label{Lemma5.1} \cite{wy3655} The $L^2$ projection operators $Q_h$, ${\cal Q}^{k-1}_h$ and ${\cal Q}^{s}_h$ satisfy the following commutative diagram:
\begin{equation}\label{l}
\nabla_{w}(Q_h w) = {\cal Q}_h^{k-1}( \nabla w), \qquad \forall  w\in H^1(T);
\end{equation}
\begin{equation}\label{div}
{\cal L}_{w}(Q_h w) = {\cal Q}_h^{s}( {\cal L} w),   \qquad \forall w \in H^1(T), \; a\nabla w \in H(div; T).
\end{equation}
\end{lemma}
\begin{proof}
For any $\boldsymbol{\psi}\in  [P_{k-1}(T)]^d$ and  $w\in H^1(T)$, from (\ref{disgradient}) and the usual integration by parts, we have
\begin{equation*}
\begin{split}
(\nabla_{w } (Q_h w), \boldsymbol{\psi})_T&=-(Q_0w,\nabla \cdot
\boldsymbol{\psi})_T+\langle Q_bw, \boldsymbol{\psi} \cdot
\textbf{n}\rangle_{\partial T}\\
&=-( w,\nabla \cdot \boldsymbol{\psi})_T+\langle w, \boldsymbol{\psi}\cdot
\textbf{n}\rangle_{\partial T}\\
&=(\nabla w, \boldsymbol{\psi})_T\\
&=({\cal Q}^{k-1}_h(\nabla w), \boldsymbol{\psi})_T,
\end{split}
\end{equation*}
which verifies (\ref{l}).

Next, for any $\phi\in P_{s}(T)$ and  $w \in H^1(T)$ such that $a\nabla w \in H(div; T)$, from (\ref{disvergence}) and the usual integration by parts, we have
\begin{equation*}
\begin{split}
({\cal L}_{w }(Q_h w), \phi)_T&= (Q_0w, {\cal L} \phi)_T-\langle Q_bw,  a\nabla\phi \cdot \bn \rangle_{\partial T} +\langle Q_n (a\nabla w\cdot \bn), \phi\rangle_{\partial T}  \\
&=(w, {\cal L} \phi)_T-\langle w,  a\nabla\phi \cdot \bn \rangle_{\partial T} +\langle  a\nabla w\cdot \bn, \phi\rangle_{\partial T} \\
&= ({\cal L} w, \phi)_T\\
&= ( {\cal Q}^s_h({\cal L} w), \phi)_T,
\end{split}
\end{equation*}
which completes the proof of (\ref{div}).
\end{proof}

\section{Solution Existence, Uniqueness, and Stability} \label{Section:ExistenceUniqueness}
The stabilizer $s(\cdot, \cdot)$ induces a semi-norm in the weak finite element space $W_h$ as follows:
\begin{equation}\label{EQ:triplebarnorm}
\3bar w \3bar = s(w, w)^{\frac{1}{2}}, \qquad \forall w\in W_h.
\end{equation}

In what follows of this paper, for the convenience of analysis, we assume that the convection vector $\bb$ and the diffusion tensor $a$ are piecewise constants associated with the finite element partition $\T_h$. However, all the analysis can be generalized to the case that the convection $\bb$ and diffusion $a$ are uniformly piecewise smooth functions without any difficulty.

\begin{lemma}\label{lem3-new} ({\it inf-sup} condition) Assume the convection tensor $\bb$ and diffusion tensor $a$ are uniformly piecewise constants with respect to the finite element partition $\T_h$. The following {\it inf-sup} condition holds true:
\begin{eqnarray} \label{EQ:inf-sup-condition-01}
 \sup_{ \lambda \in W_h^0} \frac{b(v, \lambda)}{\3bar \lambda \3bar}  & \geq & \beta_0  h^{\epsilon}  \|v\|_{\epsilon}, \quad \forall v \in M_{h},
\end{eqnarray}
where $\beta_0>0$ is a constant independent of the meshsize $h$.
\end{lemma}

\begin{proof} For any $\psi\in H^{-\epsilon}(\Omega)$,
  let $w\in H^{2-\epsilon}(\Omega)$ be the solution to the dual
  problem~\eqref{EQ:dual-form} satisfying the regularity
  estimate~\eqref{EQ:H2Regularity1}
%% \begin{eqnarray}\label{pro-new}
%% {\cal L}w_\psi +\bb \cdot \nabla w_\psi& = & \psi, \qquad \text{in}\ \Omega,\\
%% w_\psi & = & 0, \qquad\text{on} \ \Gamma_D,\label{pro0-new}\\
%% a\nabla w_\psi \cdot \bn & = & 0, \qquad\text{on} \ \Gamma_N.\label{pro0-N}
%% \end{eqnarray}
%% We assume that the solution $w:=w_\psi$ of the auxiliary problem (\ref{pro-new})-(\ref{pro0-N}) has the following $H^{2-\epsilon}$-regularity estimate
%% \begin{equation}
%%  \|w\|_{2-\epsilon} \lesssim   \|\psi\|_{-\epsilon},
%% \end{equation}
%% where $0 \leq \epsilon < \frac{1}{2}$.
By letting $\rho=Q_hw \in W_{h}^0$, from the trace inequality (\ref{tracein}) with $\theta=0$, and the estimate (\ref{error1}), we arrive at
\begin{equation}\label{EQ:Estimate:002}
\begin{split}
&\sum_{T\in {\cal T}_h }h_T^{-3}\int_{\partial
T}(|a|_T+|\bb \cdot \bn|)(\rho_0-\rho_b)^2ds\\
\lesssim &\sum_{T\in {\cal T}_h}
h_T^{-3}\int_{\partial T}|Q_0 w-Q_b w|^2ds\\
\lesssim &\sum_{T\in {\cal T}_h} h_T^{-3 }\int_{\partial T}|Q_0w-w|^2ds\\
\lesssim & \sum_{T\in {\cal T}_h} h_T^{-4}\int_{T}|Q_0w-w|^2dT
+h_T^{-2}\int_{T}|\nabla Q_0w-\nabla w|^2dT\\
\lesssim & h^{-2\epsilon}\|w\|_{2-\epsilon}^2.
 \end{split}
\end{equation} 
Analogously, we have from the trace inequality \eqref{tracein} with $\theta=\epsilon$ that
 \begin{equation}\label{EQ:Estimate:003}
\begin{split}
&\sum_{T\in {\cal T}_h }h_T^{-1}\int_{\partial T}|a\nabla
\rho_0\cdot\bn-\rho_n|^2ds\\
\lesssim & \sum_{T\in {\cal T}_h }h_T^{-1}\int_{\partial T}|a\nabla Q_0w\cdot\bn-Q_n(a\nabla w\cdot\bn)|^2ds\\
\lesssim & \sum_{T\in {\cal T}_h }h_T^{-1}\int_{\partial T}|a\nabla Q_0w\cdot\bn- a\nabla w\cdot\bn|^2ds\\
\lesssim & \sum_{T\in {\cal T}_h }h_T^{-2}\|a\nabla
Q_0w\cdot\bn- a\nabla w\cdot\bn\|_T^2+h_T^{-2\epsilon}\|a\nabla
Q_0w\cdot\bn- a\nabla w\cdot\bn\|_{1-\epsilon, T}^2\\
\lesssim & h^{-2\epsilon}\|w\|_{2-\epsilon}^2.
\end{split}
\end{equation}
The usual inverse inequality can be employed to give
\begin{equation}\label{EQ:Estimate:004}
\begin{split}
 \sum_{T\in {\cal T}_h} \gamma  ({\cal L}\rho_0+\bb \cdot \nabla \rho_0,
{\cal L}\rho_0+\bb \cdot \nabla \rho_0)_T  \lesssim & \gamma \sum_{T\in {\cal T}_h } h^{-2\epsilon}_T\|\rho_0\|^2_{2-\epsilon, T}\\
 \lesssim & \gamma \sum_{T\in {\cal T}_h }h^{-2\epsilon}_T \|Q_0 w\|^2_{2-\epsilon, T}\\
 \lesssim & \gamma h^{-2\epsilon} \|w\|^2_{2-\epsilon}.\\
 \end{split}
\end{equation}
By combining the estimates (\ref{EQ:Estimate:002})-(\ref{EQ:Estimate:004}), the $H^{2-\epsilon}$-regularity estimate (\ref{EQ:H2Regularity1}), and the definition of $\3bar \rho \3bar$, we arrive at
\begin{equation}\label{EQ:inf-sup-condition-02}
\3bar \rho\3bar \lesssim h^{-\epsilon} \|\psi\|_{-\epsilon}.
\end{equation}

Now, by letting $\rho=Q_hw \in W_h^0$, and then using Lemma \ref{Lemma5.1} we obtain
\begin{equation}\label{EQ:April05:100}
\begin{split}
b(v, \rho) = & \sum_{T\in {\cal T}_h}(v, {\cal L}_w (Q_hw)+\bb \cdot \nabla_w (Q_hw))_T  \\
 =&\sum_{T\in {\cal T}_h}(v, {\cal Q}^s_h ({\cal L} w))_T+(\bb v, {\cal Q}_h^{k-1}(\nabla w))_T
 \\
=&\sum_{T\in {\cal T}_h}(v,  {\cal L} w)_T+(\bb v,  \nabla w)_T
\\
=&\sum_{T\in {\cal T}_h}(v,  {\cal L} w + \bb \cdot \nabla w)_T
\\
=&(v, \psi).
\end{split}
\end{equation}
Using (\ref{EQ:April05:100}) and (\ref{EQ:inf-sup-condition-02}) gives
 
\begin{equation*}\label{termone}
\begin{split}
\sup_{ \lambda \in W_h^0} \frac{b(v, \lambda)}{\3bar \lambda \3bar} &\geq
\sup_{ \rho=Q_h w \in W_h^0} \frac{b(v, \rho)}{\3bar \rho \3bar} \\
& = \sup_{\psi\in H^{-\epsilon}(\Omega)} \frac{ (v, \psi)}{\3bar \rho(\psi)\3bar} \\
&\geq \beta_0 \sup_{\psi\in H^{-\epsilon}(\Omega)} \frac{(v, \psi) }{h^{-\epsilon}\|\psi\|_{-\epsilon} }\\&=\beta_0 h^{\epsilon}  \|v\|_{\epsilon}
\end{split}
\end{equation*}
for a constant $\beta_0$ independent of the meshsize $h$. This completes the proof of the lemma.
\end{proof}

\begin{remark}
If the stabilizer $s_T(\lambda, w)$ defined in (\ref{EQ:sT-form}) is chosen depending on the regularity as follows 
\begin{equation*}\label{EQ:sT-form2}
\begin{split}
s_T(\lambda, w)=&h_T^{-3+2\epsilon}\langle (|a|_T+|\bb \cdot \bn|) (\lambda_0-\lambda_b), w_0-w_b\rangle_{\partial T}
\\&+h_T^{-1+2\epsilon}\langle  a \nabla \lambda_0 \cdot \bn-\lambda_n,  a \nabla w_0 \cdot \bn-w_n\rangle_{\partial T}\\
&+\gamma h_T^{2\epsilon}({\cal L} \lambda_0+\bb \cdot \nabla \lambda_0,  {\cal L} w_0+\bb \cdot \nabla w_0)_T,\\
\end{split}
\end{equation*}
 then the inf-sup condition (\ref{EQ:inf-sup-condition-01}) is independent of $h$.
\end{remark}

The following theorem is concerned with the main result on solution existence and uniqueness for the primal-dual weak Galerkin scheme \eqref{32}-\eqref{2}.

\begin{theorem}\label{thmunique1} Assume that the diffusion tensor $a=a(\bx)$ and the convection vector $\bb$ are piecewise constants with respect to the finite element partition $\T_h$. Under the $H^{2-\epsilon}  (0 \leq \epsilon < \frac{1}{2})$-regularity assumption (\ref{EQ:H2Regularity1}), the primal-dual weak Galerkin finite element algorithm (\ref{32})-(\ref{2}) has one and only one solution for any $k\ge 2$ and $s=k-2$ or $s=k-1$ when $\gamma>0$. For the case of $\gamma=0$ (i.e., no residual stability), the numerical scheme (\ref{32})-(\ref{2}) has one and only one solution for any $k\ge 2$ and $s=k-1$.
\end{theorem}

\begin{proof} It suffices to show that zero is the only solution to the problem (\ref{32})-(\ref{2}) with homogeneous data $f=0$, $g_1=0$ and $g_2=0$. To this end, assume $f=0$, $g_1=0$ and $g_2=0$ in (\ref{32})-(\ref{2}). By letting $v=u_h$ and $w=\lambda_h$, the difference of (\ref{2}) and (\ref{32}) gives $s(\lambda_h,\lambda_h)=0$, which implies $\lambda_0=\lambda_b$ and $a\nabla \lambda_0 \cdot \bn=\lambda_n$ on each $\partial T$. This, together with the fact that $\lambda_h\in W_h^0$, leads to $\lambda_0=0$ on $\Gamma_D$ and $a\nabla \lambda_0 \cdot \bn=0$ on $\Gamma_N$.

Next, it follows from (\ref{2}), (\ref{disgradient*}), (\ref{disvergence*}) and the usual integration by parts that for all $v\in M_h$
\begin{equation}
\begin{split}
0=&b(v, \lambda_h)\\
=&\sum_{T\in {\cal T}_h}(v,  {\cal L}_w \lambda_h+\bb\cdot \nabla_w \lambda_h)_T \\
=&\sum_{T\in {\cal T}_h}  ({\cal L}\lambda_0, v)_T+\langle \lambda_0- \lambda_b, a \nabla v\cdot
\textbf{n}\rangle_{\partial T}- \langle a\nabla \lambda_0\cdot \bn-\lambda_n, v \rangle_{\partial T}
\\
&+(\nabla \lambda_0, \bb v)-\langle \lambda_0-\lambda_b, \bb v\cdot \bn \rangle_{\partial T}\\
 =&\sum_{T\in {\cal T}_h}  ( {\cal L}\lambda_0+\bb \cdot \nabla \lambda_0, v)_T+\langle \lambda_0- \lambda_b, (a \nabla v-\bb v)\cdot
\textbf{n}\rangle_{\partial T}\\
&- \langle a\nabla \lambda_0\cdot \bn-\lambda_n, v \rangle_{\partial T}\\
 =&\sum_{T\in {\cal T}_h}  ( {\cal L}\lambda_0+\bb \cdot \nabla \lambda_0, v)_T,
\end{split}
\end{equation}
where we have used $\lambda_0=\lambda_b$ and $a\nabla \lambda_0 \cdot \bn=\lambda_n$ on each $\partial T$. This implies ${\cal L}\lambda_0+\bb \cdot \nabla \lambda_0=0$ on each $T\in{\cal T}_h$ by taking $v= {\cal L}\lambda_0+\bb \cdot \nabla \lambda_0$ if $s=k-1$. For the case of $s=k-2$, from $\gamma>0$ and the fact that $s(\lambda_h,\lambda_h)=0$ we have ${\cal L}\lambda_0+\bb \cdot \nabla \lambda_0=0$ on each element $T\in\T_h$. Since $\lambda_0=0$ on $\Gamma_D$ and $a\nabla \lambda_0 \cdot \bn=0$ on $\Gamma_N$, we then have $\lambda_0\equiv 0$ in $\Omega$. It follows that $\lambda_h\equiv0$, as $\lambda_b=\lambda_0$ and $\lambda_n=a\nabla \lambda_0 \cdot \bn$ on each $\partial T$.

To show that $u_h\equiv 0$, we use $\lambda_h\equiv 0$ and the equation (\ref{32}) to obtain
\begin{equation}\label{rew}
b(u_h, w) = 0,\qquad \forall w\in W_{h}^0.
\end{equation}
From Lemma \ref{lem3-new}, we have
$$
\sup_{ w\in W_h^0}\frac{b(u_h, w)}{\3bar w \3bar} \ge \beta_0 h^{\epsilon} \|u_h\|_{\epsilon} ,
$$
which, combined with (\ref{rew}), gives $u_h\equiv 0$ in $\Omega$. This completes the proof of the theorem.
\end{proof}

\section{Error Equations}\label{Section:error-equations}
Let $u$ and $(u_h, \lambda_h) \in  M_{h} \times W_{h}^0$ be the solution of (\ref{weakform}) and its discretization scheme (\ref{32})-(\ref{2}), respectively. Note that $\lambda_h$ approximates the trivial function $\lambda=0$ as the Lagrange multiplier.

\begin{lemma}\label{Lemma:LocalEQ} Assume that the confusion tensor $a=a(\bx)$ and the convection vector $\bb$ are piecewise constant functions in $\Omega$ with respect to the finite element partition ${\cal T}_h$. For any $\sigma\in W_{h}$ and $v\in M_{h}$, the following identity holds true:
\begin{equation}
({\cal L}_w \sigma+\bb \cdot \nabla_w \sigma, v)_T = ({\cal L} \sigma_0+\bb \cdot \nabla \sigma_0, v)_T +R_T(\sigma, v),
\end{equation}
where
\begin{equation}\label{EQ:April:06:100}
\begin{split}
R_T(\sigma,v) = \langle \sigma_0-\sigma_b, (a \nabla  v-\bb v)\cdot
\textbf{n}\rangle_{\partial T}-\langle a\nabla \sigma_0\cdot \bn-\sigma_n,    v \rangle_{\partial T}.
\end{split}
\end{equation}
\end{lemma}

\begin{proof}
From (\ref{disgradient*}) and (\ref{disvergence*}), we have
\begin{equation*}\label{EQ:April:06:200}
\begin{split}
 &({\cal L}_w \sigma+\bb \cdot \nabla_w \sigma, v)_T\\
 =&(\nabla_w\sigma, \bb v)_T+({\cal L}_w \sigma,  v)_T
\\
=& (\nabla \sigma_0, \bb v)-\langle \sigma_0-\sigma_b, \bb v\cdot \bn\rangle_{\partial T}+({\cal L}\sigma_0,   v)_T\\&
+\langle \sigma_0-\sigma_b, a \nabla   v\cdot
\textbf{n}\rangle_{\partial T}-\langle a\nabla \sigma_0\cdot \bn-\sigma_n,  v \rangle_{\partial T}\\
= & ({\cal L} \sigma_0+\bb \cdot\nabla \sigma_0, v)_T +R_T(\sigma,v),
\end{split}
\end{equation*}
where $R_T(\sigma,v)$ is given by (\ref{EQ:April:06:100}).
\end{proof}

By error functions we mean the difference between the numerical solution arising from (\ref{32})-(\ref{2}) and the $L^2$ projection of the exact solution of (\ref{weakform}); i.e.,
\begin{align}\label{error}
e_h&=u_h-{\cal Q}^{s}_hu,\\
 \varepsilon_h&=\lambda_h-Q_h\lambda=\lambda_h.\label{error-2}
\end{align}

\begin{lemma}\label{errorequa}
Let $u$ and $(u_h;\lambda_h) \in M_{h}\times W_{h}^0$ be the solutions arising from (\ref{model}) and (\ref{32})-(\ref{2}), respectively. Assume that the diffusion tensor $a=a(\bx)$ and the convection vector $\bb$ are piecewise constant functions in $\Omega$ with respect to the finite element partition ${\cal T}_h$. Then, the error functions $e_h$ and $\varepsilon_h$ satisfy the following equations
\begin{eqnarray}\label{sehv}
 s( \varepsilon_h , w)+b(e_h, w)&=&\ell_u(w)
,\qquad \forall\ w\in W_{h}^0,\\
b(v, \varepsilon_h)&=&0,\qquad\qquad \forall v\in M_{h}, \label{sehv2}
\end{eqnarray}
where $\ell_u(w)$ is given by
\begin{equation}\label{lu}
\begin{split}
\qquad \ell_u(w) =&\sum_{T\in {\cal T}_h}   ({\cal L}w_0+\bb \cdot \nabla w_0, u-{\cal Q}_h^su)_T \\
 &+\langle w_0-w_b, (a \nabla ( u-{\cal Q}_h^su)-\bb (u-{\cal Q}_h^su))\cdot
\textbf{n}\rangle_{\partial T} \\
& -\langle a\nabla w_0\cdot \bn-w_n, u-{\cal Q}_h^su\rangle_{\partial T}.
\end{split}
\end{equation}
\end{lemma}

\begin{proof} From (\ref{error-2}) and (\ref{2}) we have
\begin{align*}
 b(v, \varepsilon_h) = b(v, \lambda_h) = 0,\qquad \forall v\in M_{h},
 \end{align*}
which gives rise to (\ref{sehv2}).

Next, observe that $\lambda=0$. Thus, from (\ref{32}) we arrive at
\begin{equation}\label{EQ:April:04:100}
\begin{split}
 & s(\lambda_h-Q_h\lambda, w)+b(u_h-{\cal Q}^{s}_hu, w) \\
 = & s(\lambda_h, w)+b(u_h, w)- b({\cal Q}^s_hu, w) \\
  = &-(f, w_0)+ \langle g_2, w_b \rangle_{\Gamma_N}+\langle g_1, w_n\rangle_{\Gamma_D}-b({\cal Q}^s_hu, w).
\end{split}
\end{equation}
For the term $b({\cal Q}^s_hu, w)$, we use Lemma \ref{Lemma:LocalEQ} to obtain
\begin{equation}\label{EQ:April:04:103}
\begin{split}
&b({\cal Q}^s_hu, w) \\
= & \sum_{T\in {\cal T}_h} ({\cal Q}^s_hu,  {\cal L}_w w+\bb \cdot \nabla_w w)_T \\
=& \sum_{T\in {\cal T}_h} ({\cal L}w_0+\bb \cdot \nabla w_0, {\cal Q}_h^su )_T +R_T(w, {\cal Q}_h^su) \\
= & \sum_{T\in {\cal T}_h} ({\cal L}w_0+\bb \cdot \nabla w_0, u)_T + ({\cal L}w_0+\bb \cdot \nabla w_0, {\cal Q}_h^su - u)_T +R_T(w, {\cal Q}_h^su).
\end{split}
\end{equation}
From the usual integration by parts we have
\begin{equation}\label{EQ:April:04:104}
\begin{split}
 &\sum_{T\in {\cal T}_h} ({\cal L}w_0+\bb \cdot \nabla w_0, u)_T \\= &
 \sum_{T\in {\cal T}_h}  (w_0,  \nabla \cdot (a\nabla u-\bb u))_T-\langle w_0, (a\nabla u-\bb u)\cdot\bn \rangle_{\partial T}+\langle a\nabla w_0\cdot \bn, u\rangle_{\partial T}.\\
 \end{split}
\end{equation}
Since $u$ is the exact solution of (\ref{model}), $w_b=0$ on $\Gamma_D$ and $w_n=0$ on $\Gamma_N$, we have
\begin{eqnarray}\label{EQ:April:04:105}
 \sum_{T\in {\cal T}_h}\langle w_b, (a\nabla u-\bb u)\cdot\bn \rangle_{\partial T} & = &- \langle w_b, g_2 \rangle_{\Gamma_N},\\
\sum_{T\in {\cal T}_h}  \langle w_n, u\rangle_{\partial T}& = &
\langle w_n, g_1\rangle_{\Gamma_D}.\label{EQ:April:04:106}
\end{eqnarray}
Using (\ref{EQ:April:04:104}), (\ref{EQ:April:04:105}), (\ref{EQ:April:04:106}) and (\ref{model}), we arrive at
\begin{equation}\label{EQ:April:04:107}
\begin{split}
 &\sum_{T\in {\cal T}_h} ({\cal L}w_0+\bb \cdot \nabla w_0, u)_T \\
 = & -(w_0, f)-\sum_{T\in {\cal T}_h} \langle w_0-w_b, (a\nabla u-\bb u)\cdot\bn \rangle_{\partial T}+\langle a\nabla w_0\cdot \bn-w_n, u\rangle_{\partial T}\\
 & +\langle w_b, g_2 \rangle_{\Gamma_N}+\langle w_n, g_1\rangle_{\Gamma_D}.
\end{split}
\end{equation}
Substituting (\ref{EQ:April:04:107}) and (\ref{EQ:April:04:103}) into (\ref{EQ:April:04:100}) gives rise to the error equation (\ref{sehv}), which completes the proof of the lemma.
\end{proof}

\section{Error Estimates}\label{Section:Stability}
For simplicity and without loss of generality
we introduce a semi-norm $\3bar \cdot \3bar_{b}$ in the finite element space $W_h$. For any $v=\{v_0,v_b,v_n\}\in W_h$, define on each $T\in\T_h$
\begin{equation}\label{EQ:sbT-form}
\3bar v\3bar_{b,T}:=\langle (|a|_T+|\bb\cdot\bn|)(v_0-v_b), v_0-v_b\rangle_\pT^{1/2}
\end{equation}
and
\begin{equation}\label{EQ:sb-form}
\3bar v\3bar_{b}:=\left(\sum_{T\in\T_h}  h_T^{-3} \3bar v\3bar_{b,T}^2\right)^{\frac12}.
\end{equation}
It follows from \eqref{EQ:triplebarnorm}, \eqref{EQ:s-form}, and \eqref{EQ:sT-form} that the following holds true:
\begin{equation}\label{EQ:sb-form-ineq}
\3bar v\3bar_{b}\leq \3bar v\3bar,\qquad v\in W_h.
\end{equation}

\begin{lemma}\label{Lemma:8.1}
Let ${\cal T}_h$ be a finite element partition of $\Omega$ satisfying the shape regular condition described in \cite{wy3655}. For $0\leq t \leq \min(2,k)$, the following estimates hold true:
\begin{eqnarray}\label{error1}
& \sum_{T\in {\cal T}_h}h_T^{2t}\|u-Q_0u\|^2_{t,T} \lesssim
h^{2(m+1)}\|u\|^2_{m+1},&\qquad m\in [t-1,k],\ k\ge 1,\\
\label{error2}
& \sum_{T\in {\cal T}_h}h_T^{2t}\|u-{\cal Q}^{(k-1)}_hu\|^2_{t,T} \lesssim h^{2m}\|u\|^2_{m},&\qquad m\in [t, k],\ k\ge 1, \\
\label{error3} & \sum_{T\in {\cal T}_h}h_T^{2t}\|u-{\cal
Q}^{(k-2)}_hu\|^2_{t,T}  \lesssim h^{2m}\|u\|^2_{m},&\qquad m \in [t, k-1],\ k\ge 2.\label{term3}
\end{eqnarray}
\end{lemma}

\begin{theorem} \label{theoestimate}
  Let $u$ be the solution of (\ref{weakform}) and $(u_h, \lambda_h)
  \in M_{h} \times W_{h}^0$ be its numerical solution arising from
  (\ref{32})-(\ref{2}) with index $k\ge 2$ and $s=k-2$ or
  $s=k-1$. Assume that the diffusion tensor $a=a(\bx)$ and the
  convection vector $\bb$ are piecewise constant functions in $\Omega$
  with respect to the finite element partition ${\cal T}_h$ which is
  shape regular \cite{wy3655}. Furthermore, assume that the exact
  solution $u$ is sufficiently regular such that $u\in \prod_{T\in
    T_h} H^{s+1}(T)\cap H^{2-\epsilon}(T)$ and that the regularity
  estimate~\eqref{EQ:H2Regularity1} holds for the dual
  problem~\eqref{EQ:dual-form}. Then, the following error estimate holds
  true:
 \begin{equation}\label{erres}
\3bar \lambda_h \3bar+ h^{\epsilon} \|e_h\|_{\epsilon} \lesssim
\left\{
\begin{array}{lr}
h^{s+2}(|a|^{\frac12} h^{-1} + 1 +\delta_{s,k-2}\gamma^{-\frac12})\|\nabla^{s+1} u\|, &\mbox{if } s\ge 1,\\
 h^{2} (|a|^{\frac12} h^{-1} + 1+\gamma^{-\frac12})(\|\nabla u\|+h^{\frac12-\epsilon}\|u\|_{2-\epsilon}), &\mbox{if } s=0,
\end{array}
\right.
\end{equation}
where $\delta_{i,j}$ is the Kronecker delta with value $1$ when $i=j$ and $0$ otherwise.
\end{theorem}

\begin{proof} By letting $w= \varepsilon_h=\{\varepsilon_0,\varepsilon_b,\varepsilon_n\}$ in (\ref{sehv}) and using
(\ref{sehv2}) we arrive at
\begin{equation}\label{EQ:April7:001}
s(\varepsilon_h, \varepsilon_h) = \ell_u(\varepsilon_h),
\end{equation}
where, by \eqref{lu},
\begin{equation}\label{lu-es}
\begin{split}
\qquad \ell_u(\varepsilon_h) =&\sum_{T\in {\cal T}_h}   ({\cal L}\varepsilon_0+\bb \cdot \nabla \varepsilon_0, u-{\cal Q}_h^su)_T \\
 &+\langle \varepsilon_0-\varepsilon_b, (a \nabla ( u-{\cal Q}_h^su)-\bb (u-{\cal Q}_h^su))\cdot
\textbf{n}\rangle_{\partial T} \\
& +\langle \varepsilon_n - a\nabla \varepsilon_0\cdot \bn, u-{\cal Q}_h^su\rangle_{\partial T}\\
=& \sum_{T\in {\cal T}_h} \left(J_1(T) + J_2(T)+J_3(T)\right).
\end{split}
\end{equation}
Here $J_i(T)$ is given by the corresponding term in the summation formula for $i=1,2,3$. The rest of the proof is focused on the estimate for each $J_i(T)$.

\underline{\em $J_1(T)$-estimate:} \ We recall that $s=k-2$ or $s=k-1$ is the degree of polynomials for approximating the primal variable $u$. As ${\cal L}\varepsilon_0+\bb \cdot \nabla \varepsilon_0$ is a polynomial of degree $k-1$ on each element $T$, it follows that $J_1(T)=0$ when $s=k-1$. For the case of $s=k-2$, one may use the Cauchy-Schwarz inequality to obtain
\begin{equation}\label{EQ:J1}
\begin{split}
|J_1(T)| = & \left|({\cal L}\varepsilon_0+\bb \cdot \nabla \varepsilon_0, u-{\cal Q}_h^su)_T\right|\\
\leq & \left| ({\cal L}\varepsilon_0+\bb \cdot \nabla \varepsilon_0, u-{\cal Q}_h^su)_T\right|\\
\leq & \| {\cal L}\varepsilon_0+\bb \cdot \nabla \varepsilon_0\|_T \|u-{\cal Q}_h^su\|_T\\
\lesssim & h_T^{s+1} \|\nabla^{s+1} u\|_T \| {\cal L}\varepsilon_0+\bb \cdot \nabla \varepsilon_0\|_T,
\end{split}
\end{equation}
where we have used the following interpolation error estimate in the last line:
\begin{equation}\label{EQ:Interpo-EE}
 \|u-{\cal Q}_h^su\|_T \le C h^{s+1} \|\nabla^{s+1} u\|_T.
\end{equation}
By summing \eqref{EQ:J1} over all $T\in\T_h$ we have from \eqref{EQ:triplebarnorm}, \eqref{EQ:s-form}, and \eqref{EQ:sT-form} that
\begin{equation}%\label{EQ:J2:002}
\sum_{T\in\T_h}|J_1(T)|\lesssim
\left\{
\begin{array}{lr}
\gamma^{-\frac12} h^{s+1} \|\nabla^{s+1} u\| \3bar\varepsilon_h\3bar,&\quad \mbox{for } s=k-2,\\
0,&\quad \mbox{for } s=k-1.
\end{array}
\right. \label{EQ:J1:002}
\end{equation}

\underline{\em $J_2(T)$-estimate:} From the usual Cauchy-Schwarz inequality and the boundedness of the convective vector $\bb$ we arrive at
\begin{equation}\label{EQ:J2}
\begin{split}
&|J_2(T)| \\
= & \left|\langle \varepsilon_0-\varepsilon_b, (a \nabla ( u-{\cal Q}_h^su)-\bb (u-{\cal Q}_h^su))\cdot
\textbf{n}\rangle_{\partial T}\right|\\
\le & \left| \langle \varepsilon_0-\varepsilon_b, a \nabla ( u-{\cal Q}_h^su)\cdot\textbf{n}\rangle_\pT \right| + \left| \langle \varepsilon_0-\varepsilon_b,(u-{\cal Q}_h^su)\bb\cdot
\textbf{n}\rangle_{\pT}\right|\\
\lesssim & |a|_T \|\varepsilon_0-\varepsilon_b\|_\pT \|\nabla ( u-{\cal Q}_h^su)\|_\pT + \||\bb\cdot\bn|^{\frac12}(\varepsilon_0-\varepsilon_b)\|_\pT \| u-{\cal Q}_h^su\|_\pT\\
\lesssim & \left(|a|_T^{\frac12} \|\nabla ( u-{\cal Q}_h^su)\|_\pT + \|u-{\cal Q}_h^su\|_\pT \right)\3bar \varepsilon_h\3bar_{b,T}.
\end{split}
\end{equation}
The boundary integral $\|u-{\cal Q}_h^su\|_\pT$ can be handled by using the trace inequality \eqref{tracein} and the estimate \eqref{EQ:Interpo-EE} as follows
\begin{equation}\label{EQ:188}
  \|u-{\cal Q}_h^su\|_\pT \lesssim h_T^{s+\frac12}\|\nabla^{s+1}u\|_T.
\end{equation}
As to the term $\|\nabla ( u-{\cal Q}_h^su)\|_\pT$, for $s\ge 1$, from the error estimate for the $L^2$ projection ${\cal Q}_h^su$ and the trace inequality \eqref{tracein} we have
\begin{equation}\label{EQ:189}
\|\nabla ( u-{\cal Q}_h^su)\|_\pT \lesssim h_T^{s-\frac12}\|\nabla^{s+1} u\|_T.
\end{equation}
For $s=0$, the above estimate must be modified by using the trace inequality \eqref{tracein} with $\theta=\epsilon$ as follows
\begin{equation}\label{EQ:190}
\|\nabla ( u-{\cal Q}_h^su)\|_\pT \lesssim h_T^{-\frac12}\|\nabla u\|_T + h_T^{\frac12-\epsilon}\|\nabla u\|_{1-\epsilon,T}.
\end{equation}

Next, by substituting \eqref{EQ:188}-\eqref{EQ:190} into \eqref{EQ:J2} we have
\begin{equation*}%\label{EQ:J2:002}
|J_2(T)|\lesssim
\left\{
\begin{array}{lr}
h_T^{s+\frac12}(|a|_T^{\frac12} h_T^{-1} + 1)\|\nabla^{s+1} u\|_T \3bar\varepsilon_h\3bar_{b,T},&\quad \mbox{for } s\ge 1,\\
h_T^{\frac12}(|a|_T^{\frac12} h_T^{-1} + 1)(\|\nabla u\|_T+h_T^{\frac12-\epsilon}\|\nabla u\|_{1-\epsilon,T}) \3bar\varepsilon_h\3bar_{b,T},&\quad \mbox{for } s=0,
\end{array}
\right.
\end{equation*}
Summing over $T\in\T_h$ and then using the Cauchy-Schwarz inequality and \eqref{EQ:sb-form} gives
\begin{equation}%\label{EQ:J2:002}
\sum_{T\in\T_h}|J_2(T)|\lesssim
\left\{
\begin{array}{lr}
 h^{s+2} (|a|^{\frac12} h^{-1} + 1)\|\nabla^{s+1} u\| \3bar\varepsilon_h\3bar_b, &\mbox{if } s\ge 1,\\
h^{2}(|a|^{\frac12} h^{-1} + 1)(\|\nabla u\|+h^{\frac12-\epsilon}\|u\|_{2-\epsilon})  \3bar\varepsilon_h\3bar_b, &\mbox{if } s=0.
\end{array}
\right. \label{EQ:J2:002}
\end{equation}

\underline{\em $J_3(T)$-estimate:} From the Cauchy-Schwarz and the trace inequality \eqref{tracein} we obtain
\begin{equation}\label{EQ:J3}
\begin{split}
|J_3(T)| = & \left| \langle \varepsilon_n - a\nabla \varepsilon_0\cdot \bn, u-{\cal Q}_h^su\rangle_{\partial T}\right|\\
\le & \|\varepsilon_n-a\nabla \varepsilon_0\cdot \bn \|_\pT \|u-{\cal Q}_h^su\|_\pT\\
\lesssim & \|\varepsilon_n-a\nabla \varepsilon_0\cdot \bn \|_\pT \left(h_T^{-1}\|u-{\cal Q}_h^su\|_T^2 + h_T \|\nabla(u-{\cal Q}_h^su)\|_T^2 \right)^{1/2}\\
\lesssim & h_T^{s+\frac12} \|\varepsilon_n-a\nabla \varepsilon_0\cdot \bn \|_\pT   \|\nabla^{s+1} u\|_T.
\end{split}
\end{equation}
Summing over all the element $T\in \T_h$ yields
\begin{equation}\label{EQ:J3-002}
\begin{split}
\sum_{T\in {\cal T}_h} |J_3(T)|\lesssim & \sum_{T\in {\cal T}_h} h_T^{s+\frac12} \|\varepsilon_n-a\nabla \varepsilon_0\cdot \bn \|_\pT   \|\nabla^{s+1} u\|_T\\
\lesssim & h^{s+1} \|\nabla^{s+1} u\|\left(\sum_{T\in {\cal T}_h}  h_T^{-1}  \|\varepsilon_n-a\nabla \varepsilon_0\cdot \bn \|_\pT^2\right)^{1/2}\\
\lesssim & h^{s+1} \|\nabla^{s+1} u\| \3bar \varepsilon_h\3bar.
\end{split}
\end{equation}

By combining \eqref{lu-es} with the estimates \eqref{EQ:J1:002}, \eqref{EQ:J2:002}, and \eqref{EQ:J3-002} we arrive at
\begin{equation*}%\label{aij}
|\ell_u(\varepsilon_h)|\lesssim \left\{
\begin{array}{lr}
 h^{s+2} (|a|^{\frac12} h^{-1} + 1 +\delta_{s,k-2}\gamma^{-\frac12})\|\nabla^{s+1} u\| \3bar\varepsilon_h\3bar, &\mbox{for } s\ge 1,\\
 h^{2} (|a|^{\frac12} h^{-1} + 1+\gamma^{-\frac12})(\|\nabla u\|+h^{\frac12-\epsilon}\|u\|_{2-\epsilon})  \3bar\varepsilon_h\3bar, &\mbox{for } s=0,
\end{array}
\right.
\end{equation*}
where $\delta_{i, j}$ is the Kronecker delta with value $1$ for $i=j$ and $0$ otherwise. Substituting the above estimate into (\ref{EQ:April7:001}) yields
\begin{equation*}%\label{aij}
\3bar \varepsilon_h \3bar^2\lesssim \left\{
\begin{array}{lr}
 h^{s+2} (|a|^{\frac12} h^{-1} + 1 +\delta_{s,k-2}\gamma^{-\frac12})\|\nabla^{s+1} u\| \3bar\varepsilon_h\3bar, &\mbox{for } s\ge 1,\\
 h^{2} (|a|^{\frac12} h^{-1} + 1+\gamma^{-\frac12})(\|\nabla u\|+h^{\frac12-\epsilon}\|u\|_{2-\epsilon})  \3bar\varepsilon_h\3bar, &\mbox{for } s=0,
\end{array}
\right.
\end{equation*}
which leads to
\begin{equation}\label{EQ:April7:002}
\3bar \varepsilon_h \3bar  \lesssim
\left\{
\begin{array}{lr}
 h^{s+2} (|a|^{\frac12} h^{-1} + 1 +\delta_{s,k-2}\gamma^{-\frac12})\|\nabla^{s+1} u\|, &\mbox{for } s\ge 1,\\
 h^{2} (|a|^{\frac12} h^{-1} + 1+\gamma^{-\frac12})(\|\nabla u\|+h^{\frac12-\epsilon}\|u\|_{2-\epsilon}), &\mbox{for } s=0.
\end{array}
\right.
\end{equation}
Furthermore, the error equation (\ref{sehv}) yields
$$
b(e_h, w) = \ell_u(w)-s(\varepsilon_h, w), \qquad \forall w\in W_h^0.
$$
It follows that
\begin{equation*}
\begin{split}
|b(e_h, w)| \leq & |\ell_u(w)| + \3bar \varepsilon_h\3bar  \3bar w\3bar\\
\lesssim & \left\{
\begin{array}{lr}
 h^{s+2} (|a|^{\frac12} h^{-1} + 1 +\delta_{s,k-2}\gamma^{-\frac12})\|\nabla^{s+1} u\| \3bar w \3bar, &\mbox{for } s\ge 1,\\
 h^{2} (|a|^{\frac12} h^{-1} + 1+\gamma^{-\frac12})(\|\nabla u\|+h^{\frac12-\epsilon}\|u\|_{2-\epsilon})  \3bar w \3bar, &\mbox{for } s=0,
\end{array}
\right.
\end{split}
\end{equation*}
for all $w\in W_h^0$. Thus, from the \emph{inf-sup}
condition (\ref{EQ:inf-sup-condition-01}) we obtain
\begin{equation*}
\begin{split}
\beta_0   h^{\epsilon}\|e_h\|_{\epsilon} &\lesssim
\left\{
\begin{array}{lr}
 h^{s+2} (|a|^{\frac12} h^{-1} + 1 +\delta_{s,k-2}\gamma^{-\frac12})\|\nabla^{s+1} u\|, &\mbox{for } s\ge 1,\\
 h^{2}(|a|^{\frac12} h^{-1} + 1+\gamma^{-\frac12})(\|\nabla u\|+h^{\frac12-\epsilon}\|u\|_{2-\epsilon}), &\mbox{for } s=0,
\end{array}
\right.
\end{split}
\end{equation*}
which, together with the error estimate (\ref{EQ:April7:002}), completes the proof of the theorem.
\end{proof}

From the usual triangle inequality and the error estimate (\ref{erres}), we have the following estimate for the numerical approximation of the primal variable.

\begin{corollary} \label{ErrorEstimate:4Primal:002}
Under the assumptions of Theorem \ref{theoestimate}, one has the following optimal order error estimate in the $H^{\epsilon}$-norm for $\epsilon\in [0, \frac{1}{2})$:
\begin{equation*}\label{erres:L2}
h^{\epsilon} \|u - u_h\|_{\epsilon} \lesssim
\left\{
\begin{array}{lr}
 h^{s+2}(|a|^{\frac12} h^{-1} + 1 +\delta_{s,k-2}\gamma^{-\frac12})\|\nabla^{s+1} u\|, &\mbox{if } s\ge 1,\\
 h^{2}(|a|^{\frac12} h^{-1} + 1+\gamma^{-\frac12})(\|\nabla u\|+h^{\frac12-\epsilon}\|u\|_{2-\epsilon}), &\mbox{if } s=0.
\end{array}
\right.
\end{equation*}
\end{corollary}

\medskip
We emphasize that for $s=k-1$ one has
$
({\cal L} \varepsilon_0+\bb \cdot \nabla \varepsilon_0, u-{\CQ}_h^{s}u)_T =0.
$
The proof of Theorem \ref{theoestimate} indicates that the following term
$$
 \gamma \int_T ({\cal L} \lambda_0+\bb \cdot\nabla \lambda_0)({\cal L} w_0+\bb \cdot\nabla w_0)dT
$$
in the stabilizer $s_T(\cdot,\cdot)$ \eqref{EQ:sT-form} is no longer needed in the PD-WG numerical scheme (\ref{32})-(\ref{2}). The corresponding error estimate can be stated as follows:
\begin{equation*}%\label{erres:L2}
 h^{\epsilon}\|u - u_h\|_{\epsilon} \lesssim
\left\{
\begin{array}{lr}
 h^{s+2} (|a|^{\frac12} h^{-1} + 1)\|\nabla^{s+1} u\|, &\mbox{if } s\ge 1,\\
 h^{2} (|a|^{\frac12} h^{-1} + 1+\gamma^{-\frac12})(\|\nabla u\|+h^{\frac12-\epsilon}\|u\|_{2-\epsilon}), &\mbox{if } s=0.
\end{array}
\right.
\end{equation*}

\section{Numerical Results}\label{Section:numerics}
This section shall report a variety of numerical results for the primal-dual weak Galerkin finite element scheme (\ref{32})-(\ref{2}) of the lowest order; i.e., $k=2$ and $s=0,1$. Our finite element partition $\T_h$ is given through a successive uniform refinement of a coarse triangulation of the domain by dividing each coarse level triangular element into four congruent sub-triangles by connecting the three mid-points on its edge.

Both convex and non-convex polygonal domains are considered in the numerical experiments. The representatives of the convex domains are two squares $\Omega_1=(0,1)^2$ and $\Omega_3=(-1, 1)^2$. The non-convex domains are featured by three examples: (i) the L-shaped domain $\Omega_2$ with vertices $A_1=(0, 0)$, $A_2=(2, 0)$, $A_3=(2, 1)$, $A_4=(1, 1)$, $A_5=(1, 2)$, and $A_6=(0, 2)$; (ii) the cracked square domain $\Omega_4=(-1, 1)^2\setminus(0,1)\times{0}$ (i.e., a crack along the edge $(0, 1)\times 0$); and (iii) the L-shaped domain $\Omega_5$ with vertices $B_1=(-1, -1)$, $B_2=(1, -1)$, $B_3=(1, 0)$, $B_4=(0, 0)$, $B_5=(0, 1)$, and $B_6=(-1, 1)$.

The numerical method is based on the following configuration of the weak finite element space
$$
W_{h, 2}=\{\lambda_h=\{\lambda_0,\lambda_b, \lambda_n\}: \ \lambda_0\in P_2(T), \lambda_b\in P_2(e), \lambda_n\in P_1(e), e\subset\pT, T\in {\cal T}_h\},
$$
and
$$
M_{h, s}=\{u_h: \ u_h|_T \in P_s(T),\ \forall T\in {\cal T}_h
\}, \quad s=0\  \mbox{or}\ 1.
$$

The weak finite element space $W_{h,k}$ is said to be of $C^0$-type if for any $v=\{v_0, v_b, v_n\} \in W_{h, k}$, one has $v_b=v_0|_\pT$ on each element $T\in \T_h$. Likewise, $C^{-1}$-type elements are defined as the general case of $v=\{v_0, v_b, v_n\} \in W_{h, k}$ for which $v_b$ is completely independent of $v_0$ on the edge of each element. It is clear that $C^0$-type elements involve fewer degrees of freedom compared with the $C^{-1}$-type elements. But $C^{-1}$-type elements have the flexibility in element construction and approximation. It should be noted that, for $C^{-1}$-type elements, the unknowns associated with $v_0$ can be eliminated locally on each element in parallel through a condensation algorithm before assembling the global stiffness matrix.

For simplicity of implementation, our numerical experiments will be conducted for $C^0$-type elements; i.e., $\lambda_b=\lambda_0$ on $\pT$ for each element $T\in \T_h$. For convenience, the $C^0$-type WG element with $s=1$ (i.e., $M_{h, 1}$) will be denoted as $C^0$-$P_2(T)/P_1(\pT)/P_1(T)$. Analogously, the $C^0$-type WG element corresponding to $s=0$ (i.e., $M_{h,0}$) shall be denoted as $C^0$-$P_2(T)/P_1(\pT)/P_0(T)$.

Let $\lambda_h=\{\lambda_0, \lambda_n\}\in W_{h,2}$ and $u_h\in M_{h,s}\ (s=0,1)$ be the numerical solutions arising from (\ref{32})-(\ref{2}). To demonstrate the performance of the numerical method, the numerical solutions are compared with some appropriately-chosen interpolations of the exact solution $u$ and $\lambda$ in various norms. In particular, the primal variable $u_h$ is compared with the exact solution $u$ on each element at either the three vertices (for $s=1$) or the center (for $s=0$) -- known as the nodal point interpolation $I_h u$. The auxiliary variable $\lambda_h$ approximates the true solution $\lambda=0$, and is compared with $Q_h\lambda=0$. Thus, the error functions are respectively denoted by
$$
\varepsilon_h=\lambda_h-Q_h\lambda\equiv \{\lambda_0, \lambda_n\}, \
\
e_h=u_h-I_h u.
$$
The following norms are used to measure the error functions:
\begin{eqnarray*}
\mbox{$L^2$-norm:}\quad & &  \|e_h\|_0=\Big(\sum_{T\in {\cal
T}_h} \int_T e_h^2 dT\Big)^{\frac{1}{2}},\\
\mbox{$L^2$- norm:}\quad & &  \3bar \lambda_h\3bar_0=\Big(\sum_{T\in {\cal T}_h}
\int_T \lambda_0^2 dT\Big)^{\frac{1}{2}},\\
\mbox{Semi $H^1$-norm:}\quad  & &
\3bar \lambda_h\3bar_1=\Big(\sum_{T\in {\cal T}_h} h_T
\int_{\partial T}
\lambda_n^2 ds\Big)^{\frac{1}{2}}.
\end{eqnarray*}

Tables \ref{NE:TRI:Case1-1}-\ref{NE:TRI:Case1-4} illustrate the performance of the PD-WG finite element scheme for the test problem (\ref{model}) when the exact solution is given by $u=\sin(x)\cos(y)$ for the $C^0$-type $P_2(T)/P_1(\pT)/P_1(T)$ element on the unit square domain $\Omega_1$ and the L-shaped domain $\Omega_2$ with different boundary conditions, with stabilizer parameter $\gamma=0$. The diffusion tensor in (\ref{model}) is given by $a=[10^{-10}, 0; 0, 10^{-10}]$ and the convection tensor by $\bb=[1, 1]$ which makes it a convection-dominated diffusion problem. The right-hand side function $f$, the Dirichlet boundary data $g_1$, and the Neumann boundary data $g_2$ are chosen to match the exact solution $u$. The numerical results in Tables \ref{NE:TRI:Case1-1}-\ref{NE:TRI:Case1-4} show that the convergence rates for the error function $e_h$ are of order $r=2$ in the discrete $L^2$-norm on both the unit square domain $\Omega_1$ and the L-shaped domain $\Omega_2$. The numerical results are in great consistency with the theoretical rate of convergence for $e_h$ in the discrete $L^2$-norm on the convex domain $\Omega_1$. The computational results for the non-convex domain $\Omega_2$ outperforms the theory shown in the previous section. It is interesting to see from Tables \ref{NE:TRI:Case1-1}--\ref{NE:TRI:Case1-4} that the absolute error for the numerical solution $\lambda_0$ which approximates the exact solution $\lambda=0$ is extremely small.

\begin{table}[H]
\begin{center}
\caption{Numerical rates of convergence for the $C^0$-
$P_2(T)/P_1(\pT)/P_1(T)$ element with exact solution $u=\sin(x)\cos(y)$ on
$\Omega_1=(0,1)^2$; the diffusion tensor $a=[10^{-10}, 0; 0, 10^{-10}]$; the convection vector $\bb=[1,1]$; the stabilizer parameter $\gamma=0$; Dirichlet boundary condition.}\label{NE:TRI:Case1-1}
\begin{tabular}{|c|c|c|c|c|c|c|}
\hline
$1/h$        & $\3bar\lambda_h\3bar_0 $ & order &  $\3bar\lambda_h\3bar_{1} $  & order  &   $\|e_h\|_0$  & order  \\
\hline
1&	4.82E-13&&0.008309&&	0.07550&
\\
\hline
2	&7.82E-14&	2.622&	0.001435 	&2.533 	&0.02264 	&1.738
\\
\hline
4	&1.17E-14&	2.744&	1.48E-04&	3.274 &	0.005134	&2.141
\\
\hline
8&	1.12E-15	&3.387 	&1.18E-05	&3.657	&0.001143	&2.167
\\
\hline
16	&8.30E-17	&3.749&	8.17E-07&	3.848 &	2.67E-04	&2.099
\\
\hline
32	&5.59E-18	&3.893 	&5.35E-08	&3.931 	&6.45E-05	&2.048
\\
\hline
\end{tabular}
\end{center}
\end{table}

\begin{table}[H]
\begin{center}
\caption{Numerical rates of convergence for the $C^0$-
$P_2(T)/P_1(\pT)/P_1(T)$ element with exact solution $u=\sin(x)\cos(y)$ on
$\Omega_1=(0,1)^2$; the diffusion tensor $a=[10^{-10}, 0; 0, 10^{-10}]$; the convection vector $\bb=[1,1]$; the stabilizer parameter $\gamma=0$; Neumann boundary condition on the boundary edge $(0,1)\times\{0\}$ and Dirichlet boundary condition on other three boundary edges.}\label{NE:TRI:Case1-2}
\begin{tabular}{|c|c|c|c|c|c|c|}
\hline
$1/h$        & $\3bar\lambda_h\3bar_0 $ & order &  $\3bar\lambda_h\3bar_{1} $  & order  &   $\|e_h\|_0$  & order  \\
\hline
1	&2.63E-13	&&	0.005393 	&&	0.07722	&
\\
\hline
2	&6.61E-14&	1.991&0.001270&2.087	&0.02388	&1.693
\\
\hline
4&	1.16E-14&	2.514&	1.58E-04	&3.009&	0.005821	&2.036
\\
\hline
8	&1.69E-15&	2.776 	&1.52E-05&	3.378 	&0.001425 &2.030
\\
\hline
16	&2.54E-16&	2.731 &	1.33E-06	&3.514 &	3.55E-04&	2.005
\\
\hline
32	&3.66E-17&	2.798&	1.14E-07&3.538&	8.89E-05	&1.998
\\
\hline
\end{tabular}
\end{center}
\end{table}

\begin{table}[H]
\begin{center}
\caption{
Numerical rates of convergence for the $C^0$-
$P_2(T)/P_1(\pT)/P_1(T)$ element with exact solution $u=\sin(x)\cos(y)$ on
the L-shaped domain $\Omega_2$; the diffusion tensor $a=[10^{-10}, 0; 0, 10^{-10}]$; the convection vector
$\bb=[1,1]$; the stabilizer parameter $\gamma=0$; full Dirichlet boundary condition.}\label{NE:TRI:Case1-3}
\begin{tabular}{|c|c|c|c|c|c|c|}
\hline
$1/h$        & $\3bar\lambda_h\3bar_0 $ & order &  $\3bar\lambda_h\3bar_{1} $  & order  &   $\|u_h-I_h u\|_0$  & order  \\
\hline
1	&7.33E-13&&	0.01831 	&&	0.1836	&
\\
\hline
2	&5.23E-13	&0.4862	&0.003418 	&2.421 &	0.05162 &	1.830
\\
\hline
4&	6.80E-14&	2.943 & 3.53E-04&	3.277	&0.01181&2.128
\\
\hline
8&	9.70E-15	&2.811 &	3.27E-05&	3.432 &	0.002802 	&2.075
\\
\hline
16	&1.87E-15&	2.373&	3.10E-06&	3.399 &	6.88E-04	&2.026
 \\
\hline
\end{tabular}
\end{center}
\end{table}

\begin{table}[H]
\begin{center}
\caption{Numerical rates of convergence for the $C^0$-$P_2(T)/P_1(\pT)/P_1(T)$ element with exact solution $u=\sin(x)\cos(y)$ on the L-shaped domain $\Omega_2$; the diffusion tensor $a=[10^{-10}, 0; 0, 10^{-10}]$; the convection vector $\bb=[1,1]$; the stabilizer parameter $\gamma=0$; Neumann boundary condition on the boundary edge $(0,1)\times\{0\}$ and Dirichlet boundary condition on other boundary edges.}\label{NE:TRI:Case1-4}
\begin{tabular}{|c|c|c|c|c|c|c|}
\hline
$1/h$        & $\3bar\lambda_h\3bar_0 $ & order &  $\3bar\lambda_h\3bar_{1} $  & order  &   $\|e_h\|_0$  & order  \\
\hline
1	&2.81E-12	&&	0.03304 	&&	0.2771 &
\\
\hline
2	&5.91E-13	&2.249	&0.004297	&2.943 	&0.06903 &	2.005 \\
\hline
4	&8.11E-14	&2.866	&4.49E-04&	3.260	&0.01629 &2.083
\\
\hline
8	&1.18E-14	&2.780	&4.25E-05	&3.400	&0.003996	&2.028
\\
\hline
16	&2.12E-15	&2.481	&3.92E-06	&3.437 &	9.92E-04	&2.010
\\
\hline
\end{tabular}
\end{center}
\end{table}

Tables \ref{NE:TRI:Case2-1}--\ref{NE:TRI:Case2-2} illustrate the performance of the $C^0$-type $P_2(T)/P_1(\pT)/P_0(T)$ element for the model problem (\ref{model}) with exact solution $u=\sin(x)\sin(y)$ on different boundary conditions imposed on the unit square domain $\Omega_1$ and the L-shaped domain $\Omega_2$, respectively. The stabilizer parameter is taken to be $\gamma=0$ for the third term. The diffusion tensor is given by $a=[10^{-3}, 0; 0, 10^{-3}]$ and the convection vector by $\bb=[1, 1]$. Table \ref{NE:TRI:Case2-1} shows that the convergence for $u_h$ in the discrete $L^2$ norm is at the rate of ${\cal O}(h)$ which is consistent with what the theory predicts for the convex domain $\Omega_1$. On the L-shaped domain $\Omega_2$, the PD-WG method appears to be convergent at a rate better than ${\cal O}(h)$ which outperforms the theory prediction.

\begin{table}[H]
\begin{center}
\caption{Numerical rates of convergence for the $C^0$-$P_2(T)/P_1(\pT)/P_0(T)$ element with exact solution $u=\sin(x)\sin(y)$ on $\Omega_1$; the diffusion tensor $a=[10^{-3}, 0; 0, 10^{-3}]$; the convection vector $\bb=[1,1]$; the stabilizer parameter $\gamma=0$; full Dirichlet boundary condition.}\label{NE:TRI:Case2-1}
\begin{tabular}{|c|c|c|c|c|c|c|}
\hline
$1/h$ & $\3bar\lambda_h\3bar_0 $ & order &  $\3bar\lambda_h\3bar_{1} $  & order  &   $\|e_h\|_0$  & order  \\
\hline
1&	7.31E-09	&	&1.60E-04	&&0.06262 &
\\
\hline
2&	2.06E-04	&-14.78 &	0.002612	&-4.025 &1.761 &	-4.814
\\
\hline
4&	1.87E-05	&3.463 	&5.21E-04	&2.327 &	0.7674 &	1.198
\\
\hline
8	&1.28E-06&	3.873 &	7.37E-05	&2.821&	0.1996&1.943
\\
\hline
16	&8.43E-08&3.918 	&9.84E-06&	2.905&	0.04795 	&2.057
\\
\hline
32&	5.91E-09&	3.836	&1.37E-06&	2.850	&0.02394 	&1.002
\\
\hline
\end{tabular}
\end{center}
\end{table}

\begin{table}[H]
\begin{center}
\caption{Numerical rates of convergence for the $C^0$-
$P_2(T)/P_1(\pT)/P_0(T)$ element with exact solution $u=\sin(x)\sin(y)$ on
the L-shaped domain $\Omega_2$; the diffusion tensor $a=[10^{-3}, 0; 0, 10^{-3}]$; the convection vector $\bb=[1,1]$; the stabilizer parameter $\gamma=0$; Neumann boundary condition on the boundary edge $(0,1)\times\{0\}$ and Dirichlet boundary condition on other boundary edges.}\label{NE:TRI:Case2-2}
\begin{tabular}{|c|c|c|c|c|c|c|}
\hline
$1/h$        & $\3bar\lambda_h\3bar_0 $ & order &  $\3bar\lambda_h\3bar_{1} $  & order  &   $\|e_h\|_0$  & order  \\
\hline
1	&0.006017	&&	3.62E-02&&	10.32 	&
\\
\hline
2	&7.26E-04&	3.050 	&0.009553	&1.920 &	3.411 &	1.598
\\
\hline
4	&4.82E-05&	3.912 &	0.001289 &	2.890	&0.9021 &1.919
\\
\hline
8	&3.29E-06&	3.873 	&1.70E-04&	2.919 &	0.2906	&1.634
\\
\hline
16	&2.20E-07&3.905 &	2.21E-05	&2.947 	&0.1175&1.307
\\
\hline
\end{tabular}
\end{center}
\end{table}

Tables \ref{NE:TRI:Case9-1}-\ref{NE:TRI:Case9-2} show the numerical results with the $C^0$-$P_2(T)/P_1(\pT)/P_1(T)$ element and the $C^0$-$P_2(T)/P_1(\pT)/P_0(T)$ element on the cracked domain $\Omega_4$ when the exact solution is given by  $u=\sin(x)\sin(y)$. In this numerical experiment, we considered a convection-dominated diffusion problem in which the diffusion tensor is given by $a=[10^{-5}, 0; 0, 10^{-5}]$ and the convection vector by $\bb=[1, 0]$. The Neumann boundary condition is imposed on the inflow boundary $\{-1\}\times (-1, 1)$ (i.e., the edge where $\bb\cdot \bn<0$), and the Dirichlet boundary condition is imposed on the rest of the boundary. Tables \ref{NE:TRI:Case9-1}-\ref{NE:TRI:Case9-2} indicate that the convergence for $e_h$ in the discrete $L^2$ norm is at the rate of ${\cal O}(h^2)$ and ${\cal O}(h)$, respectively.

 \begin{table}[H]
\begin{center}
\caption{Numerical rates of convergence for the $C^0$-$P_2(T)/P_1(\pT)/P_1(T)$ element with exact solution $u=\sin(x)\sin(y)$ on cracked square domain $\Omega_4$; the diffusion tensor $a=[10^{-5}, 0; 0, 10^{-5}]$; the convection direction $\bb=[1,0]$; the stabilizer parameter $\gamma=0$; Neumann boundary condition on the inflow boundary edge $\{-1\}\times (-1, 1)$; and Dirichlet boundary condition on the rest of the boundary.}\label{NE:TRI:Case9-1}
\begin{tabular}{|c|c|c|c|c|c|c|}
\hline
$1/h$ & $\3bar\lambda_h\3bar_0 $ & order & $\3bar\lambda_h\3bar_{1} $  & order  & $\|e_h\|_0$  & order
\\
\hline
1	& 2.57E-11	& &  	1.63E-06	& & 	0.3623 	&
\\
\hline
2	& 2.51E-12& 	3.359 & 	1.09E-07& 	3.896 	& 0.09193	& 1.979
\\
\hline
4	& 3.53E-13	& 2.828 & 	7.47E-09& 	3.871 	& 0.02247& 	2.032
\\
\hline
8& 	6.00E-14	& 2.558	& 5.78E-10& 	3.691	& 0.005545 & 2.019
\\
\hline
16& 	1.30E-14	& 2.202 & 	5.26E-11	& 3.459& 0.001383 	& 2.003 \\
\hline
32&	5.20E-15&	1.325 	&7.18E-12	&2.873	&3.69E-04	&1.907 \\
\hline
\end{tabular}
\end{center}
\end{table}

\begin{table}[H]
\begin{center}
\caption{Numerical rates of convergence for the $C^0$-$P_2(T)/P_1(\pT)/P_0(T)$ element with exact solution $u=\sin(x)\sin(y)$ on cracked square domain $\Omega_4$; the diffusion tensor $a=[10^{-5}, 0; 0, 10^{-5}]$; the convection vector $\bb=[1,0]$; the stabilizer parameter $\gamma=0$; Neumann boundary condition on the inflow boundary $\{-1\}\times (-1, 1)$ and Dirichlet boundary condition on the rest of the boundary.}\label{NE:TRI:Case9-2}
\begin{tabular}{|c|c|c|c|c|c|c|}
\hline
$1/h$ & $\3bar\lambda_h\3bar_0 $ & order & $\3bar\lambda_h\3bar_{1} $  & order  & $\|e_h\|_0$  & order
\\
\hline
1	&1.08E-04	&&	0.1127 	&&	0.1693 	&
\\
\hline
2	&4.66E-04&	-2.110 	&0.03210 	&1.812 &	0.09306&	0.8634
\\
\hline
4	&3.14E-05&	3.891 	&0.008091 &	1.988 &	0.04655&0.9995
\\
\hline
8	&1.90E-06	&4.044	&0.002012 &	2.008	&0.02300 &1.017
\\
\hline
16	&1.17E-07&	4.025&	4.95E-04	&2.023 &	0.01119 	&1.039
\\
\hline
32&7.92E-09&3.883 &1.17E-04&2.081 &0.005579& 1.004
\\
\hline
\end{tabular}
\end{center}
\end{table}

Table \ref{NE:TRI:Case3-1} illustrates the performance of the PD-WG method with the $C^0$- $P_2(T)/P_1(\pT)/P_0(T)$ element when the exact solution is given by $u=\sin(x)\sin(y)$ on the unit square domain $\Omega_1$. The diffusion tensor is given by $a=[1+x^2+y^2, 0; 0, 1+x^2+y^2]$ and the convection vector by $\bb=[x, y]$. The stabilizer parameter $\gamma=0$. The convergence for $e_h$ in the discrete $L^2$ norm is at the rate of ${\cal O}(h)$ which is consistent with what the theory predicts.

\begin{table}[H]
\begin{center}
\caption{Numerical rates of convergence for the $C^0$- $P_2(T)/P_1(\pT)/P_0(T)$ element with exact solution $u=\sin(x)\sin(y)$ on
$\Omega_1$; the diffusion tensor $a=[1+x^2+y^2, 0; 0, 1+x^2+y^2]$; the convection vector$\bb=[x, y]$; the stabilizer parameter $\gamma=0$; full Dirichlet boundary condition.}\label{NE:TRI:Case3-1}
\begin{tabular}{|c|c|c|c|c|c|c|}
\hline
$1/h$        & $\3bar\lambda_h\3bar_0 $ & order &  $\3bar\lambda_h\3bar_{1} $  & order  &   $\|e_h\|_0$  & order
  \\
\hline
1&	0.02967&&	0.4979 &&0.04851 &
  \\
\hline
2&	0.002843 &	3.384&	0.1173 	&2.086&	0.02801&	0.7925   \\
\hline
4&	4.53E-04	&2.649 &	0.02797&	2.069 &	0.01272 	&1.138
  \\
\hline
8	&1.02E-04&	2.155 &0.006792 	&2.042	&0.006047 	&1.073
  \\
\hline
16	&2.45E-05&	2.053	&0.001671 &2.023 	&0.002980	&1.021
  \\
\hline
32&6.07E-06	&2.016&	4.14E-04	&2.012 &0.001485&1.005
\\
\hline
\end{tabular}
\end{center}
\end{table}

Tables \ref{NE:TRI:Case4-1}-\ref{NE:TRI:Case4-3} illustrate the numerical performance of the numerical scheme with $C^0$- $P_2(T)/P_1(\pT)/P_1(T)$ and $C^0$- $P_2(T)/P_1(\pT)/P_0(T)$ elements when the exact solution is given by $u=xy(1-x)(1-y)$ on the unit square domain $\Omega_1$. The diffusion tensor is given by $a=[10^{-5}, 0; 0, 10^{-5}]$ and the convection vector by $\bb=[1,1]$; the stabilizer parameter is set as $\gamma=0$; and various boundary conditions are considered. This is a convection-dominated diffusion problem. The numerical results in Tables \ref{NE:TRI:Case4-1}-\ref{NE:TRI:Case4-2} show that the convergence for $e_h$ in the $L^2$ norm is at a rate higher than the theoretical prediction of ${\cal O}(h^2)$. Moreover, it can be seen from Table \ref{NE:TRI:Case4-3} that the convergence for $e_h$ in the $L^2$ norm is at a rate higher than the theoretical prediction of ${\cal O}(h)$.

\begin{table}[H]
\begin{center}
\caption{Numerical rates of convergence for the $C^0$-
$P_2(T)/P_1(\pT)/P_1(T)$ element with exact solution $u=xy(1-x)(1-y)$ on
$\Omega_1$; the diffusion tensor $a=[10^{-5}, 0; 0, 10^{-5}]$; the convection vector $\bb=[1,1]$; the stabilizer parameter $\gamma=0$; full Dirichlet boundary condition.}\label{NE:TRI:Case4-1}
\begin{tabular}{|c|c|c|c|c|c|c|}
\hline
$1/h$ & $\3bar\lambda_h\3bar_0 $ & order & $\3bar\lambda_h\3bar_{1} $  & order  & $\|e_h\|_0$  & order
\\
\hline
1&1.50E-11&&	1.41E-06	&&	1.73E-06&
\\
\hline
2&	9.30E-09	&-9.276&	0.001574&	-10.12 	&0.01608 &	-13.18 \\
\hline
4	&1.20E-09&	2.949	&1.73E-04&	3.185 &	0.002986 	&2.429
\\
\hline
8	&1.24E-10	&3.275 	&1.49E-05&	3.539	&4.72E-04	&2.660
\\
\hline
16&	1.05E-11&	3.571 &	1.10E-06	&3.753 &	6.54E-05	&2.852
\\
\hline
32&	7.56E-13	&3.791 &	7.51E-08&	3.879	&8.51E-06	& 2.943
\\
\hline
\end{tabular}
\end{center}
\end{table}
\begin{table}[H]
\begin{center}
\caption{Numerical rates of convergence for the $C^0$-
$P_2(T)/P_1(\pT)/P_1(T)$ element with exact solution $u=xy(1-x)(1-y)$ on
$\Omega_1$; the diffusion tensor $a=[10^{-5}, 0; 0, 10^{-5}]$; the convection vector $\bb=[1,1]$; the stabilizer parameter $\gamma=0$; Neumann boundary condition on the edge $\{0\}\times (0,1)$ and Dirichlet boundary condition on the rest of the boundary.}\label{NE:TRI:Case4-2}
\begin{tabular}{|c|c|c|c|c|c|c|}
\hline
$1/h$ & $\3bar\lambda_h\3bar_0 $ & order & $\3bar\lambda_h\3bar_{1} $  & order  & $\|e_h\|_0$  & order
\\
\hline
1&	5.76E-08&&	8.37E-03&&	5.11E-02	&
\\
\hline
2&	1.41E-08&	2.035 	&1.70E-03	&2.301	&2.43E-02&	1.070
\\
\hline
4&	2.81E-09&	2.325 &	2.06E-04	&3.047 	&5.60E-03	&2.118
\\
\hline
8	&4.26E-10	&2.719 	&1.99E-05	&3.366 	&1.21E-03&	2.211
\\
\hline
16	&5.77E-11	&2.885 	&1.74E-06&	3.518 	&2.66E-04	&2.187
\\
\hline
32&	7.49E-12	&2.946 &1.47E-07&	3.562 &6.09E-05	&2.126
\\
\hline
\end{tabular}
\end{center}
\end{table}

\begin{table}[H]
\begin{center}
\caption{Numerical rates of convergence for the $C^0$-
$P_2(T)/P_1(\pT)/P_0(T)$ element with exact solution $u=xy(1-x)(1-y)$ on
$\Omega_1$; the diffusion tensor $a=[10^{-5}, 0; 0, 10^{-5}]$; the convection vector $\bb=[1,1]$; the stabilizer parameter $\gamma=0$; full Dirichlet boundary condition.}\label{NE:TRI:Case4-3}
\begin{tabular}{|c|c|c|c|c|c|c|}
\hline
$1/h$ & $\3bar\lambda_h\3bar_0 $ & order & $\3bar\lambda_h\3bar_{1} $  & order  & $\|e_h\|_0$  & order
\\
\hline
1	&1.10E-04	&&	0.001213	&&	0.04557	&
\\
\hline
2&	3.23E-04	&-1.550	&0.002780&	-1.197	&0.02386 	&0.9333
\\
\hline
4	&7.97E-05	&2.019 	&6.21E-04&	2.161 &	0.004707 &2.342
\\
\hline
8	&1.62E-05	&2.294 &1.06E-04	&2.552 &	0.001344 &	1.808
\\
\hline
16&	3.72E-06	&2.128&	2.16E-05&	2.297 	&4.87E-04	&1.465
\\
\hline
32	&9.05E-07	&2.038	&5.03E-06&	2.101&	2.08E-04	&1.226
\\
\hline
\end{tabular}
\end{center}
\end{table}

Tables \ref{NE:TRI:Case5-1}--\ref{NE:TRI:Case5-2} illustrate the numerical
performance of the PD-WG method with the $C^0$- $P_2(T)/P_1(\pT)/P_1(T)$ element and $C^0$-type $P_2(T)/P_1(\pT)/P_0(T)$ element for a test problem with exact solution $u=y(1-y)(1-e^{-x})(1-e^{-(1-x)})$ on the unit square domain $\Omega_1$. The configuration of the test problem is as follows: the diffusion tensor $a=[10^{-3}, 0; 0, 10^{-3}]$; the convection vector $\bb=[1,1]$; the stabilizer parameter $\gamma=0$. The numerical results in Tables \ref{NE:TRI:Case5-1}--\ref{NE:TRI:Case5-2} show that the numerical convergence is in great consistency with our theory of error estimate.

\begin{table}[H]
\begin{center}
\caption{Numerical rates of convergence for the $C^0$-
$P_2(T)/P_1(\pT)/P_1(T)$ element with exact solution $u=y(1-y)(1-e^{-x})(1-e^{-(1-x)})$ on $\Omega_1$; the diffusion tensor $a=[10^{-3}, 0; 0, 10^{-3}]$; the convection vector $\bb=[1,1]$; the stabilizer parameter $\gamma=0$; Neumann boundary condition on the edge $\{0\}\times (0, 1)$ and the Dirichlet boundary condition on the rest of the boundary.}\label{NE:TRI:Case5-1}
\begin{tabular}{|c|c|c|c|c|c|c|}
\hline
$1/h$ & $\3bar\lambda_h\3bar_0 $ & order & $\3bar\lambda_h\3bar_{1} $  & order  & $\|e_h\|_0$  & order
\\
\hline
1&3.59E-06	&&	0.005209 	&&	0.03177 	&
\\
\hline
2	&8.90E-07&	2.014&0.001060 &	2.297	&0.01522 &	1.062
\\
\hline
4&	1.75E-07&	2.346	&1.29E-04& 3.044	&0.003510 &2.117
\\
\hline
8	&2.59E-08	&2.759	&1.27E-05& 3.343 &7.59E-04	&2.209
\\
\hline
16&	3.32E-09	&2.960 & 1.26E-06	&3.325 	&1.67E-04&	2.184
\\
\hline
32&	3.96E-10	&3.068 	&1.85E-07&	2.776 &	3.88E-05	&2.108
\\
\hline
\end{tabular}
\end{center}
\end{table}

\begin{table}[H]
\begin{center}
\caption{Numerical rates of convergence for the $C^0$-
$P_2(T)/P_1(\pT)/P_0(T)$ element with exact solution  $u=y(1-y)(1-e^{-x})(1-e^{-(1-x)})$ on $\Omega_1$; the diffusion tensor $a=[1, 0; 0, 1]$; the convection vector $\bb=[1,1]$; the stabilizer parameter $\gamma=0$;  full Dirichlet boundary condition.}\label{NE:TRI:Case5-2}
\begin{tabular}{|c|c|c|c|c|c|c|}
\hline
$1/h$ & $\3bar\lambda_h\3bar_0 $ & order & $\3bar\lambda_h\3bar_{1} $  & order  & $\|e_h\|_0$  & order
\\
\hline
1	&0.002518 &&	0.01444 	&&	0.02065&
\\
\hline
2	&0.002451 	&0.03860	&0.009379&	0.6226 &0.006689	&1.626
\\
\hline
4	&6.34E-04	&1.950 &	0.002354 	&1.994	&0.001807	&1.888
\\
\hline
8&	1.51E-04&	2.067 &	5.36E-04	&2.135 &6.09E-04	&1.569
\\
\hline
16	&3.71E-05&	2.029&	1.28E-04&	2.067&	2.59E-04	&1.236
\\
\hline
32&	9.22E-06	&2.009 &	3.15E-05	&2.022 	&1.23E-04	&1.074
\\
\hline
\end{tabular}
\end{center}
\end{table}

Tables \ref{NE:TRI:Case6-1}-\ref{NE:TRI:Case6-4} show the numerical results on the unit square domain $\Omega_1$ for the $C^0$- $P_2(T)/P_1(\pT)/P_1(T)$ and $C^0$- $P_2(T)/P_1(\pT)/P_0(T)$ elements, respectively. In this numerical experiment, we consider a convection-dominated diffusion problem by taking the diffusion tensor as $a=[10^{-5}, 0; 0, 10^{-5}]$ and the convection vector $\bb=[1, 0]$. The stabilizer parameter for the third term is given by $\gamma=0$; and Dirichlet boundary data is imposed on all the boundary edges. The exact solutions are chosen to be $u=0.5(1-\tanh((x-0.5)/0.2))$ and $u=0.5(1-\tanh((x-0.5)/0.05))$, respectively. The numerical results in Tables \ref{NE:TRI:Case6-1}-\ref{NE:TRI:Case6-2} indicate that the convergence for $e_h$ in the $L^2$ norm seem to arrive at a superconvergence rate of ${\cal O}(h^2)$ which is higher than the theoretical prediction of ${\cal O}(h)$ for the $C^0$- $P_2(T)/P_1(\pT)/P_0(T)$ element. Tables \ref{NE:TRI:Case6-3}-\ref{NE:TRI:Case6-4} show that the convergence for $e_h$ in the $L^2$ norm is at the rate of ${\cal O}(h^2)$ for the $C^0$- $P_2(T)/P_1(\pT)/P_1(T)$ element which is consistent with the theoretical error estimate.

\begin{table}[H]
\begin{center}
\caption{Numerical rates of convergence for the $C^0$- $P_2(T)/P_1(\pT)/P_0(T)$ element with exact solution $u=0.5(1-\tanh((x-0.5)/0.2))$ on $\Omega_1$; the diffusion tensor $a=[10^{-5}, 0; 0, 10^{-5}]$; the convection vector $\bb=[1,0]$; the stabilizer parameter $\gamma=0$; full Dirichlet boundary condition.}\label{NE:TRI:Case6-1}
\begin{tabular}{|c|c|c|c|c|c|c|}
\hline
$1/h$ & $\3bar\lambda_h\3bar_0 $ & order & $\3bar\lambda_h\3bar_{1} $  & order  & $\|e_h\|_0$  & order
\\
\hline
1&	5.93E-11	&&	4.28E-06	&&	0.02001 	&
\\
\hline
2	&2.22E-04&	-21.83 &	0.003797 &	-9.794&	1.62E+02	&-12.98 \\
\hline
4	&2.63E-05&	3.073 	&9.09E-04&	2.063 &	59.22 &	1.453 \\
\hline
8	&1.87E-06&	3.813 &	1.31E-04&	2.799	&10.42&2.507
\\
\hline
16	&1.21E-07&	3.949 &	1.70E-05	&2.945 	&2.215 	&2.234
\\
\hline
32&	7.69E-09	&3.979	&2.15E-06	&2.977 	&0.4904 &2.175 \\
\hline
\end{tabular}
\end{center}
\end{table}

\begin{table}[H]
\begin{center}
\caption{Numerical rates of convergence for the $C^0$- $P_2(T)/P_1(\pT)/P_0(T)$ element with exact solution  $u=0.5(1-\tanh((x-0.5)/0.05))$ on $\Omega_1$; the diffusion tensor $a=[10^{-5}, 0; 0, 10^{-5}]$; the convection vector $\bb=[1,0]$; the stabilizer parameter $\gamma=0$; full Dirichlet boundary condition.}\label{NE:TRI:Case6-2}
\begin{tabular}{|c|c|c|c|c|c|c|}
\hline
$1/h$ & $\3bar\lambda_h\3bar_0 $ & order & $\3bar\lambda_h\3bar_{1} $  & order  & $\|e_h\|_0$  & order
\\
\hline
1	&1.32E-10	&&	7.02E-06	&&0.06502  &
\\
\hline
2	&1.01E-04	&-19.54 &1.66E-03&	-7.882 &	57.82	&-9.796 \\
\hline
4	&1.86E-05	&2.442	&6.28E-04	&1.399 	&29.78 &0.9571
\\
\hline
8	&2.61E-06	&2.833 &	1.80E-04	&1.803&	6.781	&2.135
\\
\hline
16&	2.32E-07&	3.488 	&3.24E-05&	2.476 	&1.725&1.975
\\
\hline
32&1.54E-08&	3.913 &	4.32E-06&	2.906 &	0.4000&	2.109\\
\hline
\end{tabular}
\end{center}
\end{table}

\begin{table}[H]
\begin{center}
\caption{Numerical rates of convergence for the $C^0$- $P_2(T)/P_1(\pT)/P_1(T)$ element with exact solution  $u=0.5(1-\tanh((x-0.5)/0.2))$ on $\Omega_1$; the diffusion tensor $a=[10^{-5}, 0; 0, 10^{-5}]$; the convection vector $\bb=[1,0]$; the stabilizer parameter $\gamma=0$; full Dirichlet boundary condition.}\label{NE:TRI:Case6-3}
\begin{tabular}{|c|c|c|c|c|c|c|}
\hline
$1/h$ & $\3bar\lambda_h\3bar_0 $ & order & $\3bar\lambda_h\3bar_{1} $  & order  & $\|e_h\|_0$  & order
\\
\hline
1	&1.54E-06	&&	0.08745 	&&	0.1581 	&
\\
\hline
2	&2.84E-07&	2.437 &	0.02606 &1.746 	&0.2589 &-0.7115
\\
\hline
4	&4.60E-08	&2.629 &	0.002860&	3.188 &	0.06301 	&2.039
\\
\hline
8&	8.99E-09&	2.356 &	2.23E-04	&3.681&	0.01346 	&2.227
\\
\hline
16	&2.30E-09	&1.963 &	1.99E-05	&3.485 &	0.003350 	&2.006
\\
\hline
32&	5.89E-10&	1.968 &	1.98E-06& 3.331 &8.42E-04	&1.992
\\
\hline
\end{tabular}
\end{center}
\end{table}

\begin{table}[H]
\begin{center}
\caption{Numerical rates of convergence for the $C^0$-
$P_2(T)/P_1(\pT)/P_1(T)$ element with exact solution  $u=0.5(1-\tanh((x-0.5)/0.05))$ on $\Omega_1$; the diffusion tensor $a=[10^{-5}, 0; 0, 10^{-5}]$; the convection vector $\bb=[1,0]$; the stabilizer parameter $\gamma=0$; full Dirichlet boundary conditions.}\label{NE:TRI:Case6-4}
\begin{tabular}{|c|c|c|c|c|c|c|}
\hline
$1/h$ & $\3bar\lambda_h\3bar_0 $ & order & $\3bar\lambda_h\3bar_{1} $  & order  & $\|e_h\|_0$  & order
\\
\hline
1	&6.49E-06	&&	0.3685	&&	0.6662 	&
\\
\hline
2	&6.84E-07&	3.246&	0.07008 &	2.394 &	0.6305 &	0.07942
\\
\hline
4&	2.24E-07	&1.614	&0.01776 &	1.980 	&0.3130 	&1.010
\\
\hline
8	&4.12E-08	&2.439	&0.003311 	&2.423 &	0.1281 	&1.289
\\
\hline
16	&5.98E-09	&2.785 &	3.58E-04&	3.209 &	0.03184	&2.009 \\
\hline
32	&1.23E-09	&2.279 	&2.81E-05&	3.671 &	0.006791&2.229
\\
\hline
\end{tabular}
\end{center}
\end{table}

Tables \ref{NE:TRI:Case7-1}-\ref{NE:TRI:Case7-2} illustrate the numerical results for the $C^0$- $P_2(T)/P_1(\pT)/P_1(T)$ and the $C^0$- $P_2(T)/P_1(\pT)/P_0(T)$ elements on the unit square domain $\Omega_1$ with exact solution $u= e^{-(x-0.5)^2/0.2-3(y-0.5)^2/0.2}$. The test problem has the diffusion tensor $a=[10^{-5}, 0; 0, 10^{-5}]$ and the convection $\bb=[1, 0]$. The stabilizer parameters are chosen as $\gamma=1$ and $\gamma=0$, respectively. The Dirichlet boundary condition is imposed on the entire boundary. The numerical results in Table \ref{NE:TRI:Case7-1} show a superconvergence for $e_h$ in the $L^2$ norm, as the optimal order error estimate would imply a convergence at the rate of ${\cal O}(h)$ when the $C^0$- $P_2(T)/P_1(\pT)/P_0(T)$ element is used. Table \ref{NE:TRI:Case7-2} indicates that the convergence order for $e_h$ in the discrete $L^2$ norm is consistent with what the theory predicts.

\begin{table}[H]
\begin{center}
\caption{Numerical rates of convergence for the $C^0$- $P_2(T)/P_1(\pT)/P_0(T)$ element with exact solution  $u= e^{-(x-0.5)^2/0.2-3(y-0.5)^2/0.2}$ on $\Omega_1$; the diffusion tensor $a=[10^{-5}, 0; 0, 10^{-5}]$; the convection vector $\bb=[1,0]$; the stabilizer parameter $\gamma=1$; Dirichlet boundary condition on the entire boundary.}\label{NE:TRI:Case7-1}
\begin{tabular}{|c|c|c|c|c|c|c|}
\hline
$1/h$ & $\3bar\lambda_h\3bar_0 $ & order & $\3bar\lambda_h\3bar_{1} $  & order  & $\|e_h\|_0$  & order
\\
\hline
1&4.06E-15	&&	8.77E-07	&&	0.4682	&
\\
\hline
2	&3.21E-04&	-36.20 &	0.005274 & -12.55 &1.21E+02	&-8.016
\\
\hline
4	&2.52E-05&	3.673 	&7.73E-04&	2.771	&7.002 &4.113
\\
\hline
8&	1.35E-06&	4.221 	&9.47E-05	&3.029&	3.980	&0.8152
\\
\hline
16	&8.44E-08&	4.000 &	1.19E-05&	2.998 	&0.8133 &2.291 \\
\hline
32&	5.26E-09	&4.004&	1.48E-06&	3.005	&0.1313	&2.631
\\
\hline
\end{tabular}
\end{center}
\end{table}

\begin{table}[H]
\begin{center}
\caption{Numerical rates of convergence for the $C^0$-
$P_2(T)/P_1(\pT)/P_1(T)$ element with exact solution  $u=e^{-(x-0.5)^2/0.2-3(y-0.5)^2/0.2}$ on $\Omega_1$; the diffusion tensor $a=[10^{-5}, 0; 0, 10^{-5}]$; the convection $\bb=[1,0]$; the stabilizer parameter $\gamma=0$; Dirichlet boundary condition on the entire boundary.}\label{NE:TRI:Case7-2}
\begin{tabular}{|c|c|c|c|c|c|c|}
\hline
$1/h$ & $\3bar\lambda_h\3bar_0 $ & order & $\3bar\lambda_h\3bar_{1} $  & order  & $\|e_h\|_0$  & order
\\
\hline
1&2.80E-10	&&	0.05239&&	0.2339 	&
\\
\hline
2	&3.74E-07	&-10.38 	&0.01817 &1.528	&0.1609 &0.5398
\\
\hline
4&7.24E-08	&2.369 	&0.003381 	&2.426 	&0.1146 	&0.4893
\\
\hline
8&	1.77E-08	&2.035 &	3.09E-04	&3.452  &0.03390 &1.757
\\
\hline
16	&4.43E-09&	1.994 &	2.93E-05	&3.398 &	0.008362 	&2.019 \\
\hline
32	&1.27E-09&	1.799 &	3.30E-06&	3.151 &	0.002117 	&1.982
\\
\hline
\end{tabular}
\end{center}
\end{table}

Figure \ref{ux} illustrates the plots of the numerical solution $u_h$ arising from the PD-WG scheme (\ref{32})-(\ref{2}) for a convection-dominated diffusion problem on the unit square domain $\Omega_1$. In this numerical experiment, the diffusion tensor is given by $a=[10^{-5}, 0; 0, 10^{-5}]$, the convection vector by $\bb=[1, 0]$, and the load function is given by $f=1$. The Neumann boundary data $g_2=10^{-5}$ is imposed on the boundary edge $\{0\}\times (0, 1)$, and the Dirichlet boundary data $g_1=x$ is imposed on the rest of the boundary edges. The figure on the left shows the numerical solution $u_h$ when the $C^0$-type $P_2(T)/P_1(\pT)/P_1(T)$ element is used and the one on the right is for the numerical solution $u_h$ with the $C^0$-type $P_2(T)/P_1(\pT)/P_0(T)$ element.

\begin{figure}[h]
\centering
\begin{tabular}{cc}
\resizebox{2.4in}{2.1in}{\includegraphics{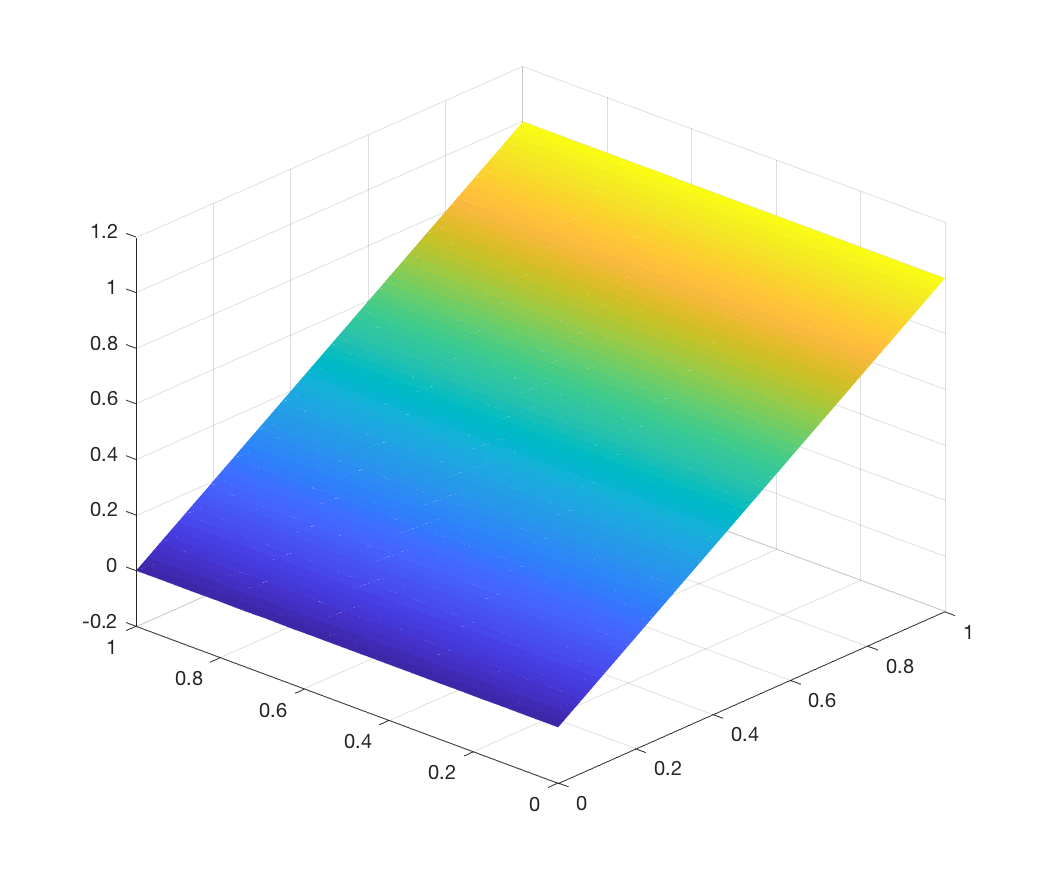}}
\resizebox{2.4in}{2.1in}{\includegraphics{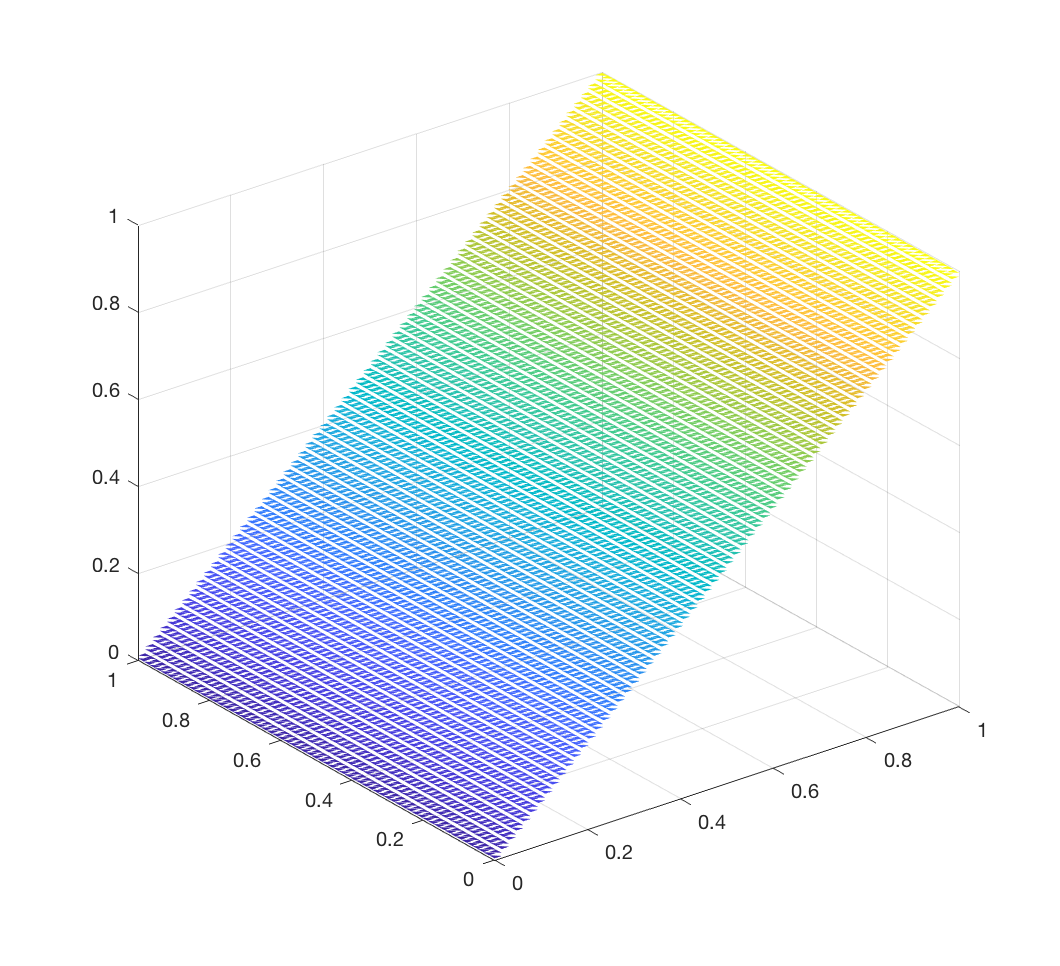}}
\end{tabular}
\caption{Surface plot of $u_h$ on the unit square domain $\Omega_1$: left for the $C^0$-$P_2(T)/P_1(\pT)/P_1(T)$ element, right for the $C^0$- $P_2(T)/P_1(\pT)/P_0(T)$ element.}
\label{ux}
\end{figure}

Figure \ref{athree} shows the plots for the numerical solution $u_h$ on the unit square domain $\Omega_1$ when the $C^0$-type $P_2(T)/P_1(\pT)/P_0(T)$ element is employed to the test problem with convective direction $\bb=[1, 0]$ and load function $f=1$. The Neumann boundary condition of $g_2=a_{11}$ (where $a=(a_{ij})$) is imposed on the inflow boundary edge $\{0\}\times (0, 1)$ and the Dirichlet boundary condition $g_1=0$ is imposed on the rest of the boundary. Figure \ref{athree} shows the numerical solution $u_h$ for different diffusion tensors: $a=[10^{-1}, 0; 0, 10^{-1}]$ (left), $a=[10^{-3}, 0; 0, 10^{-3}]$ (middle), and $a=[10^{-6}, 0; 0, 10^{-6}]$ (right).

\begin{figure}[h]
\centering
\begin{tabular}{ccc}
\resizebox{1.5in}{1.5in}{\includegraphics{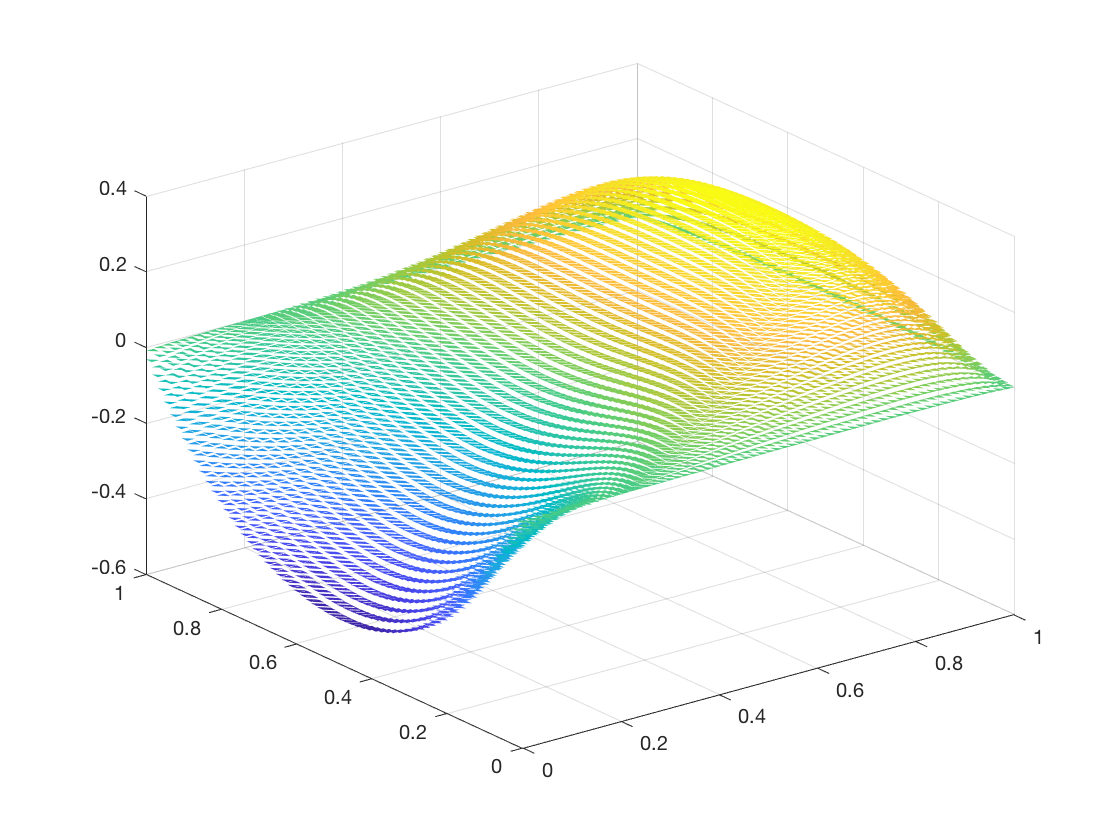}}
\resizebox{1.5in}{1.5in}{\includegraphics{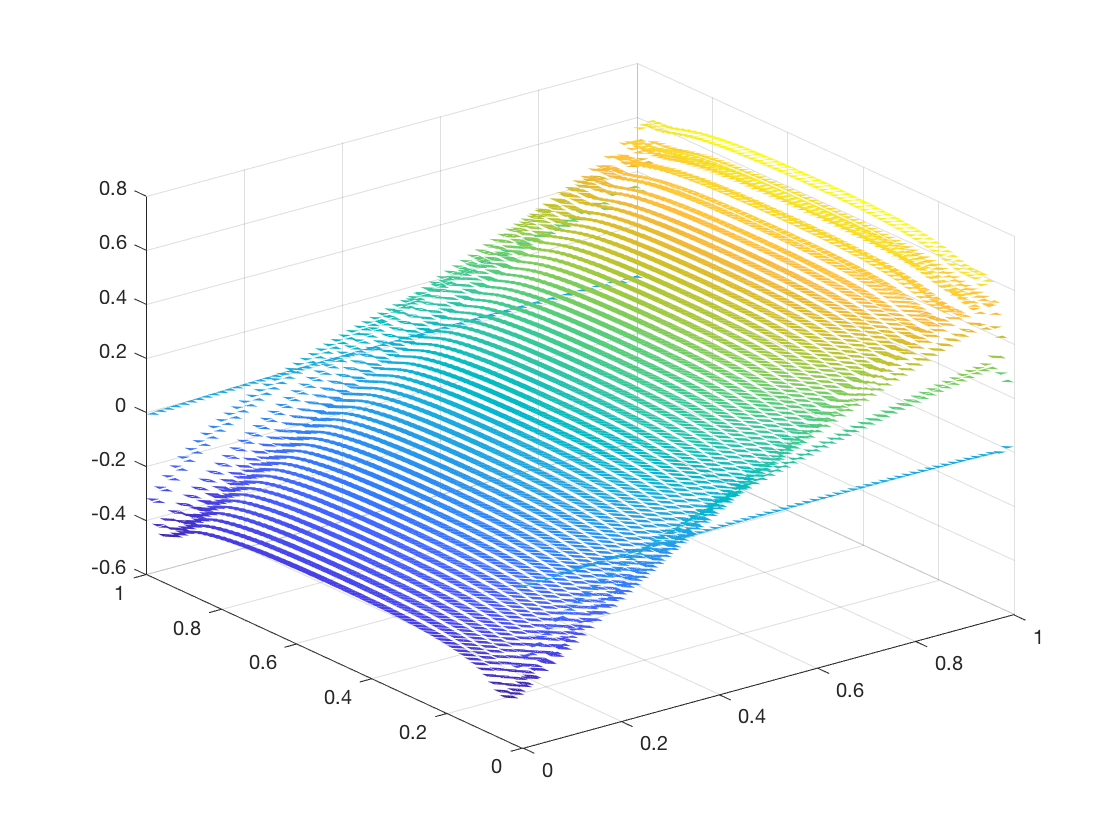}}
\resizebox{1.5in}{1.5in}{\includegraphics{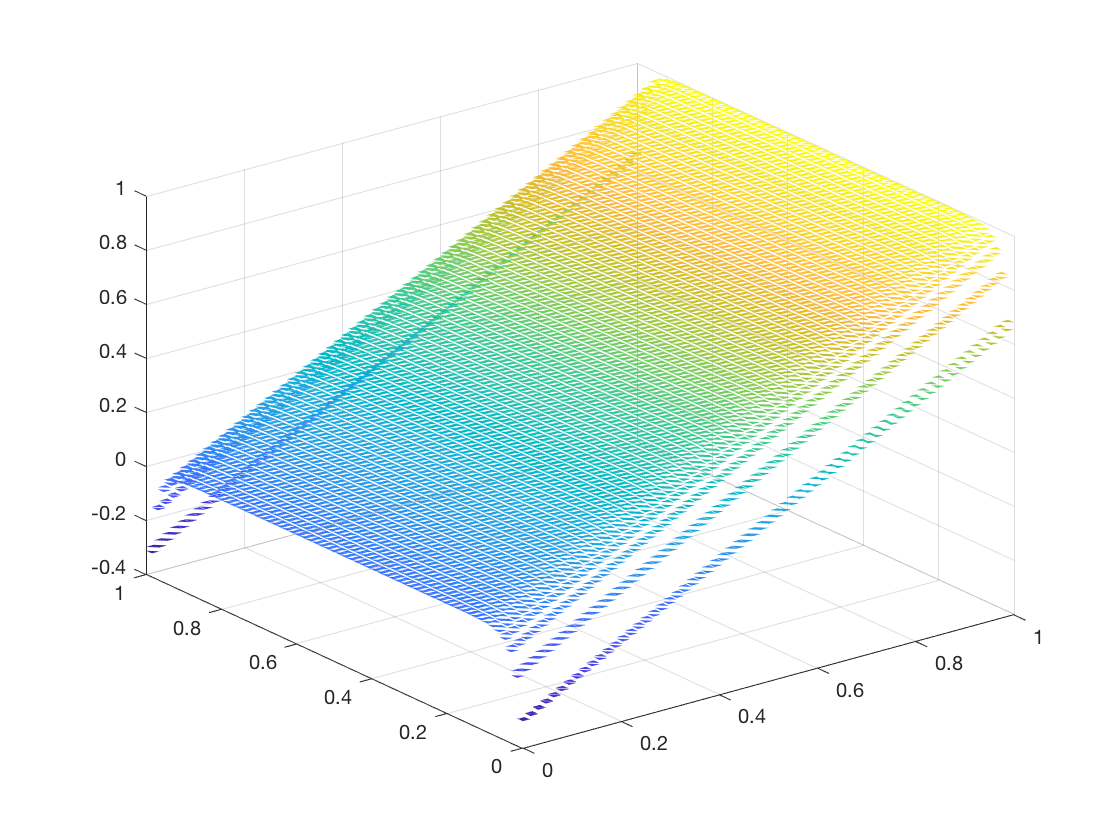}}
\end{tabular}
\caption{Surface plots for the primal variable $u_h$ on the unit square domain $\Omega_1$ with the $C^0$- $P_2(T)/P_1(\pT)/P_0(T)$ element: left for the diffusion tensor $a=[10^{-1}, 0; 0, 10^{-1}]$, middle for the diffusion tensor $a=[10^{-3}, 0; 0, 10^{-3}]$, right for the diffusion tensor $a=[10^{-6}, 0; 0, 10^{-6}]$.}
\label{athree}
\end{figure}

Figures \ref{squaredomain}-\ref{Lshapeddomain} illustrate the contour plots for the numerical solution $u_h$ arising from the primal-dual weak Galerkin finite element method on three different domains: (i) the square domain $\Omega_3=(-1, 1)^2$, (ii) the cracked square domain $\Omega_4$, and (iii) the L-shaped domain $\Omega_5$. In this numerical experiment, the model problem has a diffusion tensor $a=[10^{-4}, 0;0, 10^{-4}]$ and a convective (rotational) vector $\bb=[y, -x]$. Figures \ref{squaredomain}-\ref{Lshapeddomain} are obtained by using the following configurations: (a) the $C^0$-$P_2(T)/P_1(\pT)/P_1(T)$ element, (b) Neumann boundary condition $g_2=0$ on the inflow boundary edges ($\bb\cdot \bn<0$), and (c) Dirichlet boundary condition $g_1=sin(3x)$ on the outflow boundary edges ($\bb\cdot \bn>0$). The load functions are taken as $f=0$ and $f=1$, respectively.

\begin{figure}[h]
\centering
\begin{tabular}{cc}
  \resizebox{2.4in}{2.1in}{\includegraphics{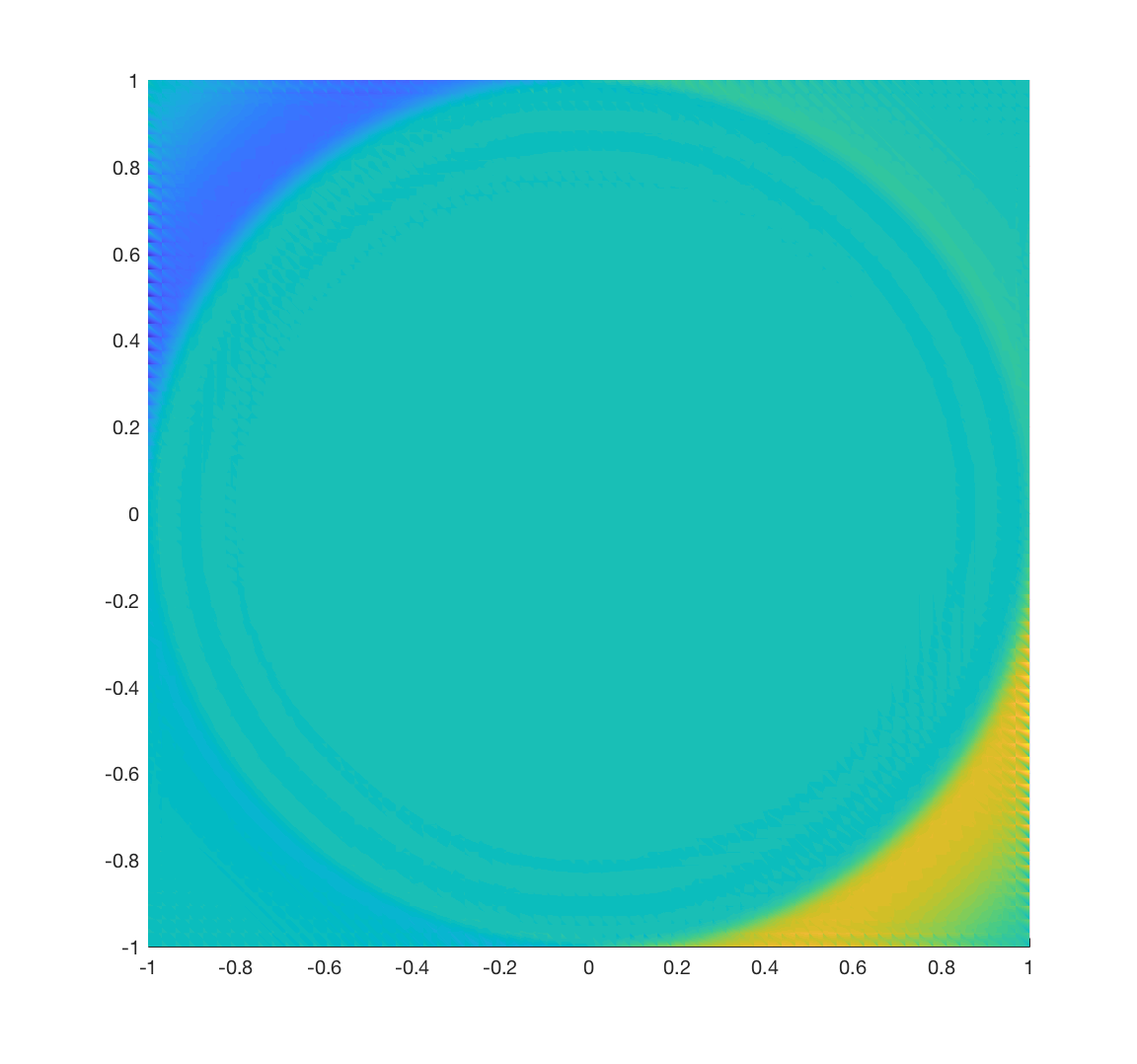}}
  \resizebox{2.4in}{2.1in}{\includegraphics{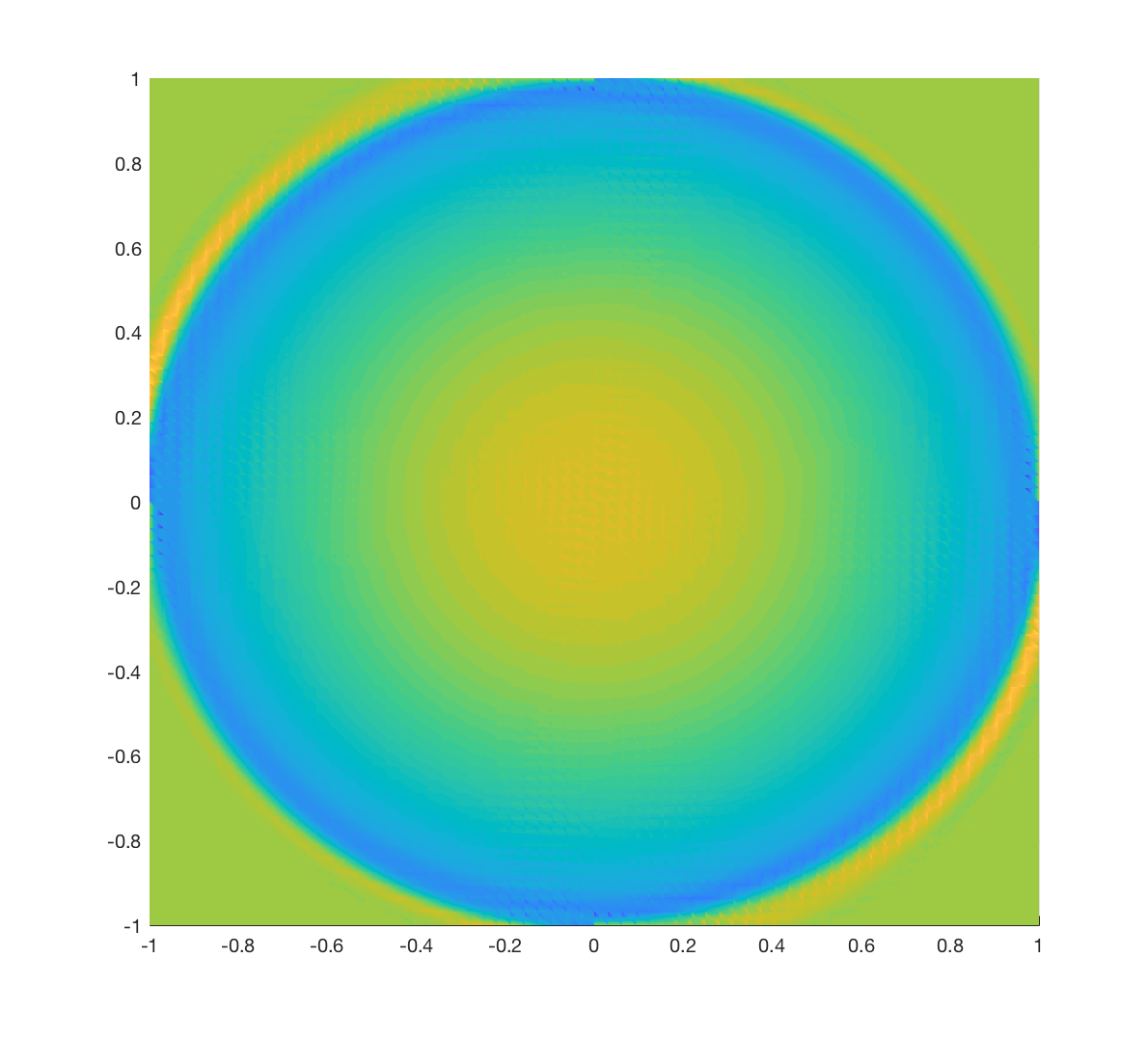}}
\end{tabular}
  \caption{Contour plots for primal variable $u_h$: load function $f=0$ (left), load function $f=1$ (right). The square domain $\Omega_3$ and the $C^0$- $P_2(T)/P_1(\pT)/P_1(T)$ element.}
\label{squaredomain}
\end{figure}

\begin{figure}[h]
\centering
\begin{tabular}{cc}
\resizebox{2.4in}{2.1in}{\includegraphics{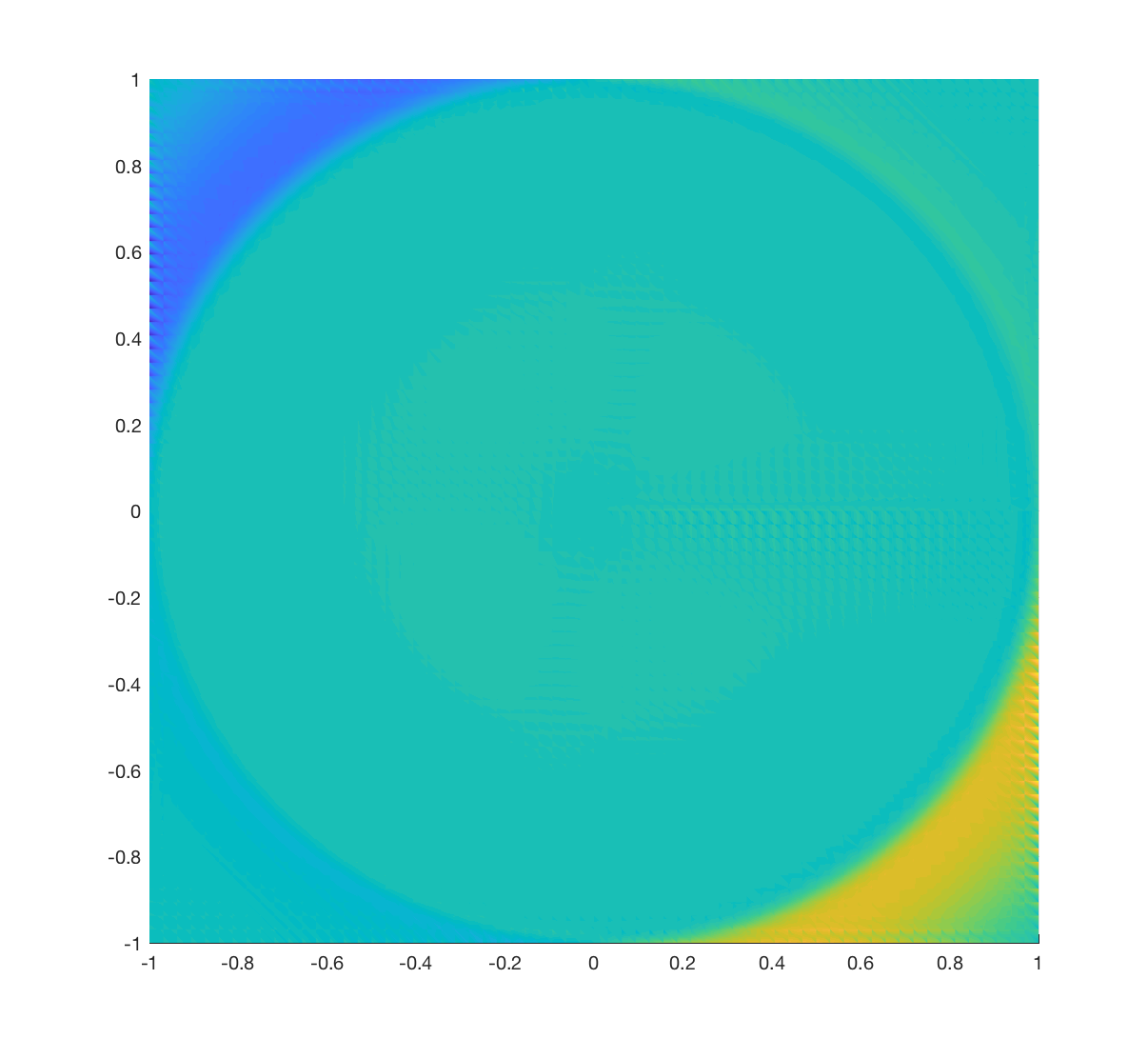}}
\resizebox{2.4in}{2.1in}{\includegraphics{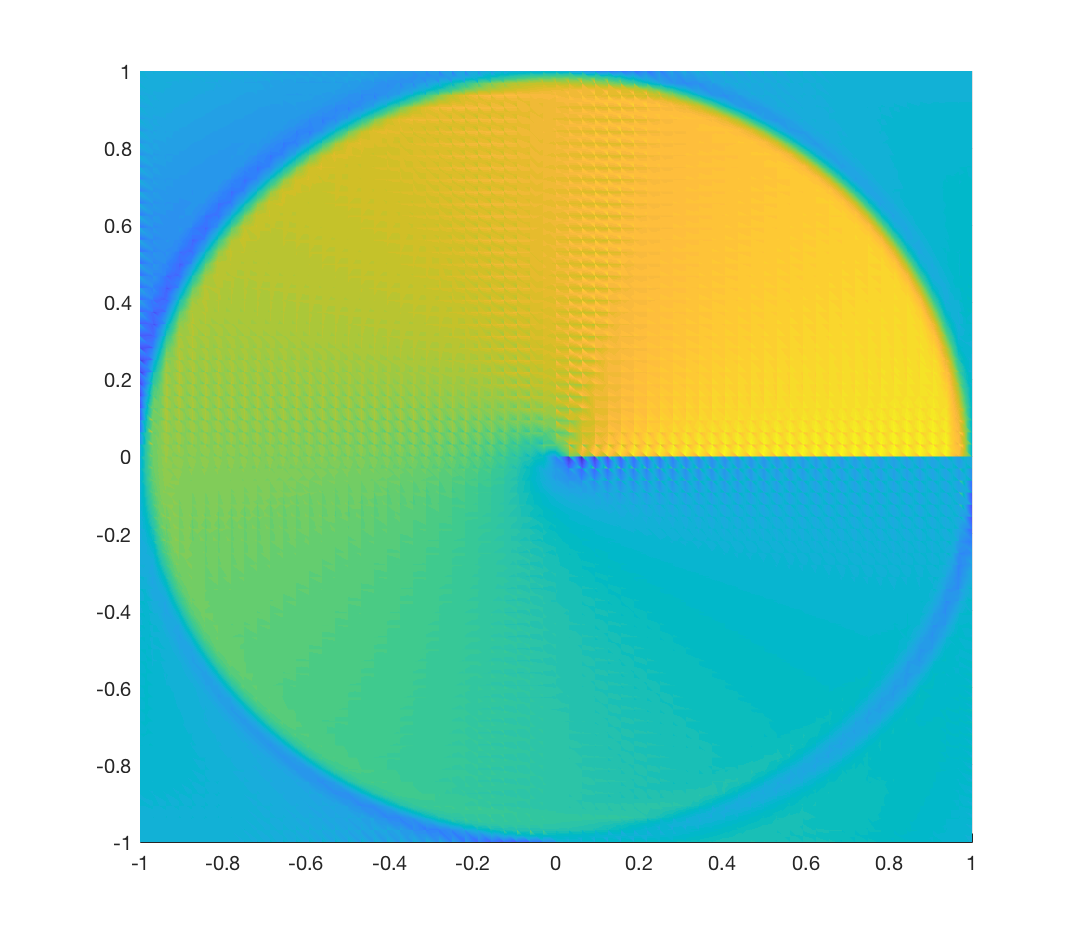}}
\end{tabular}
\caption{Contour plots for primal variable $u_h$: load function $f=0$ (left), load function $f=1$ (right). Cracked square domain $\Omega_4$ and the $C^0$- $P_2(T)/P_1(\pT)/P_1(T)$ element.}
\label{crackedsquare}
\end{figure}

\begin{figure}[h]
\centering
\begin{tabular}{cc}
\resizebox{2.4in}{2.1in}{\includegraphics{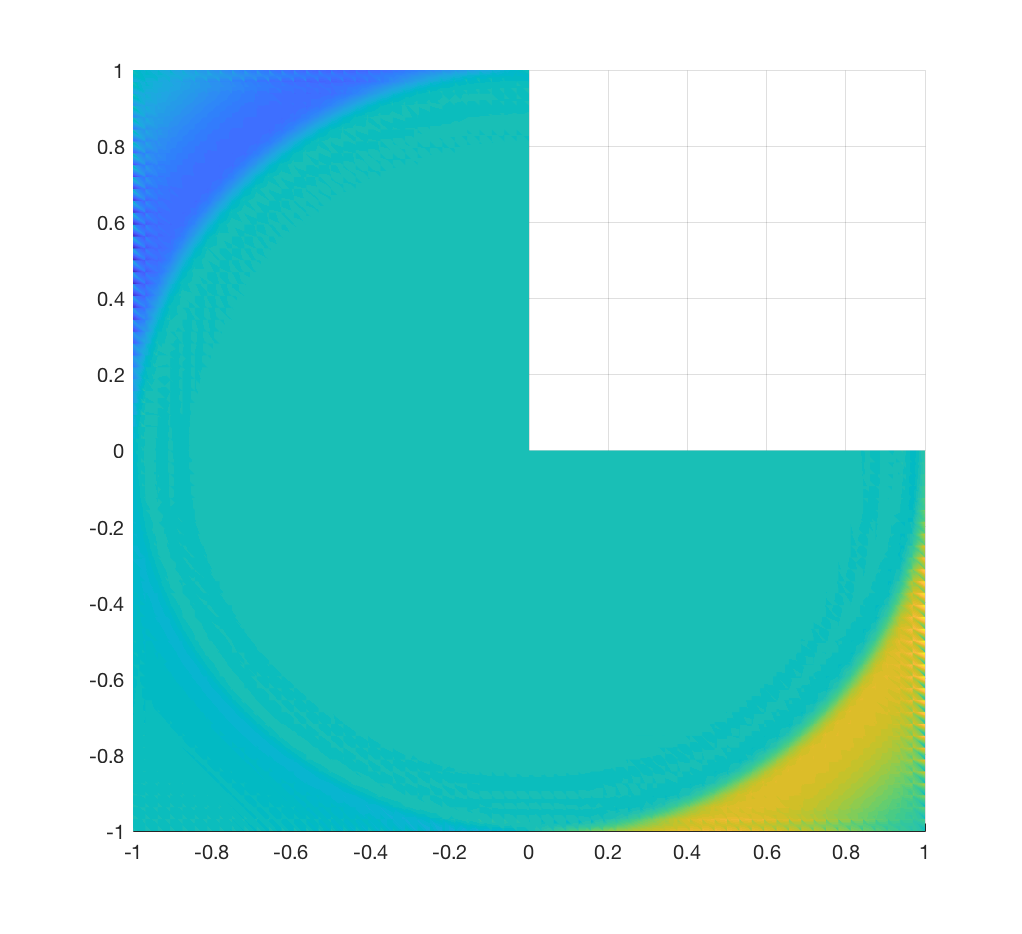}}
\resizebox{2.4in}{2.1in}{\includegraphics{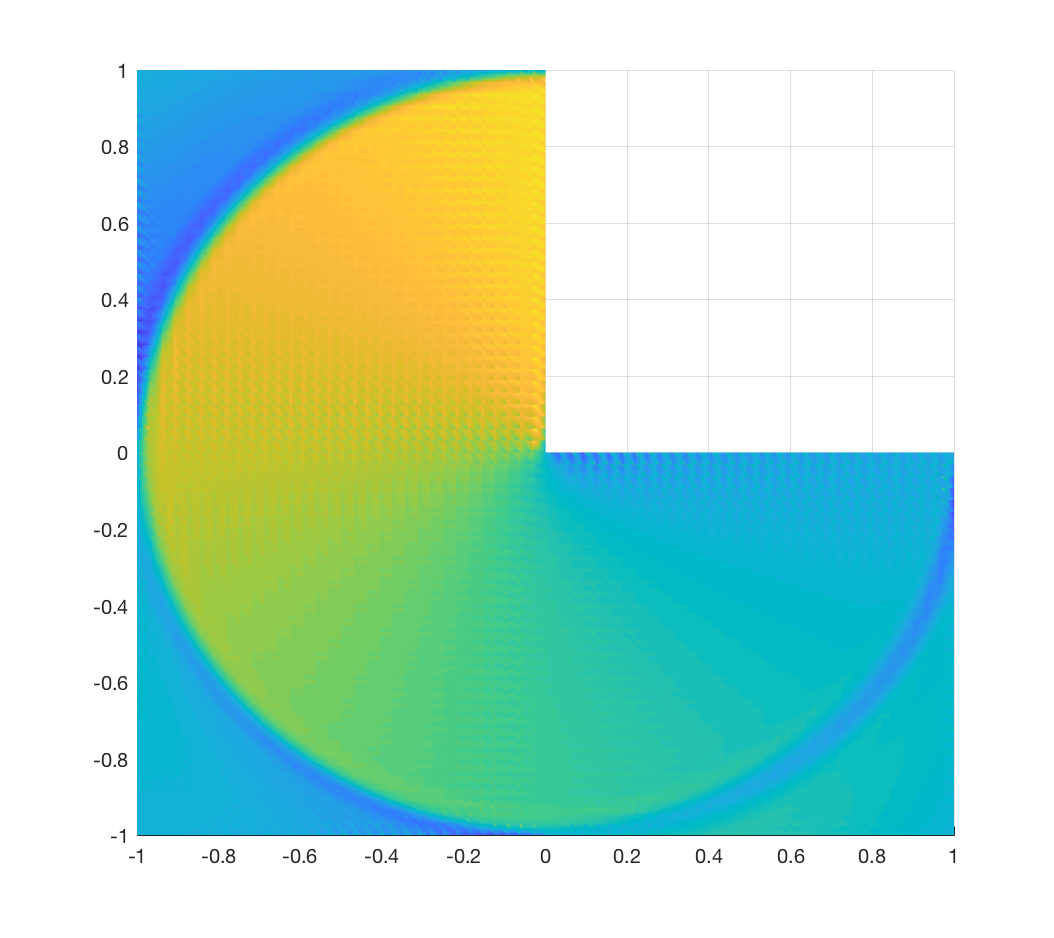}}
\end{tabular}
\caption{Contour plots for the primal variable $u_h$: load function $f=0$ (left), load function $f=1$ (right). L-shaped domain $\Omega_5$ with the $C^0$- $P_2(T)/P_1(\pT)/P_1(T)$ element. }
\label{Lshapeddomain}
\end{figure}

In summary, the numerical performance of the PD-WG scheme (\ref{32})-(\ref{2}) for the convection-dominated convection-diffusion problem \eqref{model} is typically consistent with or better than what our theory predicts. Theorem \ref{theoestimate} and the numerical tests show that the stabilization parameter $\gamma$ is not necessary to make the PD-WG method convergent and accurate when $s=k-1$. We conjecture that the PD-WG finite element scheme with $\gamma=0$ is stable and has the optimal order of convergence for both $s=k-2$ and $s=k-1$ when the diffusion tensor $a$ and the convection vector $\bb$ are uniformly piecewise continuous functions, provided that the meshsize is sufficiently small. Interested readers are encouraged to explore the corresponding theory with more sophisticated mathematical tools.

\section{Conclusions}
The primal-dual weak Galerkin finite element method developed here for
convection diffusion problems has shown several promising features as
a discretization approach in the following aspects: (1) it provides a
symmetric and well-posed discrete problem; (2) it is consistent in the
sense that the exact solution, if sufficiently regular, satisfies the
discrete variational problem; (3) it allows for low regularity of the
primal variable and admits optimal a priori error estimates.  Further
exploration is needed for constructing fast solvers for the resulting
discrete problems and this is a subject of a current and future work.

\end{document}